\documentclass{amsart}

\usepackage[foot]{amsaddr}

\usepackage[margin=3cm,includefoot,footskip=30pt]{geometry}
\usepackage[english]{babel}

\usepackage{amssymb,amsfonts,mathtools,amsthm}
\usepackage{graphicx}
\usepackage{subfig}
\usepackage{enumitem}
\setenumerate{label={{\rm({\roman*})}},leftmargin=6mm,itemsep=3pt,topsep=3pt}
\setitemize{label={$\vcenter{\hbox{\tiny$\bullet$}}$},leftmargin=6mm,itemsep=3pt,topsep=3pt}

\usepackage{lmodern}
\usepackage[algoruled,linesnumbered,lined]{algorithm2e}
\usepackage{algpseudocode}

\usepackage[outline]{contour}
\contourlength{1.3pt}
\usepackage{tikz}
\usetikzlibrary{calc,arrows}
\tikzstyle{every picture}=[line width=.65pt,minimum size=3pt,every label/.append style={font=\small},label distance=-1pt]
\usepackage{url}

\newcommand{\R}{\mathbb{R}}

\newcommand{\bb}{\mathbb}
\newcommand{\Z}{\bb Z}

\newcommand{\vertex}{\mathrm{vert}}
\newcommand{\aff}{\mathrm{aff}}
\newcommand{\rk}{\mathrm{rk}}
\newcommand{\conv}{\mathrm{conv}}
\newcommand{\cl}{\mathrm{cl}}
\newcommand{\dchcl}{\cl_{\mathrm{dch}}}
\newcommand{\inccl}{\cl_{\mathrm{inc}}}

\newcommand{\polylog}{\mathrm{polylog}}

\newcommand{\firstX}{\mathcal{X}_{\mathrm{full}}}
\newcommand{\secondX}{\mathcal{X}_{\mathrm{inc}}}
\newcommand{\finalX}{\mathcal{X}}

\newcommand{\altlexeq}{\mathop{{\preccurlyeq}\hskip-.6mm\raisebox{.45pt}{$\cdot$}}}
\newcommand{\altlex}{\mathop{{\prec}\hspace{-.6mm}\cdot}}

\newtheorem{thm}{Theorem}
\newtheorem{cor}[thm]{Corollary}
\newtheorem{lem}[thm]{Lemma}
\newtheorem{prop}[thm]{Proposition}
\newtheorem{conj}[thm]{Conjecture}

\theoremstyle{definition}
\newtheorem{defn}[thm]{Definition}
\newtheorem{ex}[thm]{Example}

\makeatletter
\def\paragraph{\medskip\@startsection{paragraph}{4}
  \z@\z@{-\fontdimen2\font}
  {\normalfont\itshape}}
\makeatother

\title{Enumeration of $2$-level polytopes}

\author[A.~Bohn]{Adam Bohn\textsuperscript{1}}
\address[1,3,4,5]{Universit\'e libre de Bruxelles, Brussels, Belgium}
\email{adam.bohn@ulb.ac.be}

\author[Y.~Faenza]{Yuri Faenza\textsuperscript{2}}
\address[2]{IEOR Department, Columbia University, New York, USA}
\email{yf2414@columbia.edu}

\author[S.~Fiorini]{Samuel Fiorini\textsuperscript{3}}
\email{sfiorini@ulb.ac.be}

\author[V.~Fisikopoulos]{Vissarion Fisikopoulos\textsuperscript{4}}
\email{vissarion.fisikopoulos@gmail.com}

\author[M.~Macchia]{Marco Macchia\textsuperscript{5}}
\email{mmacchia@ulb.ac.be}

\author[K.~Pashkovich]{Kanstantsin Pashkovich\textsuperscript{6}}
\address[6]{C \& O Department, University of Waterloo, Waterloo, Canada}
\email{kanstantsin.pashkovich@gmail.com}


\date{\today}

\begin{document}

\maketitle

\begin{abstract}
A (convex) polytope $P$ is said to be \emph{$2$-level} if for every direction of hyperplanes which is facet-defining for $P$, the vertices of $P$ can be covered with two hyperplanes of that direction. The study of these polytopes is motivated by questions in combinatorial optimization and communication complexity, among others. In this paper, we present the first algorithm for enumerating all combinatorial types of $2$-level polytopes of a given dimension $d$, and provide complete experimental results for $d \leqslant 7$. Our approach is inductive: for each fixed $(d-1)$-dimensional $2$-level polytope $P_0$, we enumerate all $d$-dimensional $2$-level polytopes $P$ that have $P_0$ as a facet. This relies on the enumeration of the closed sets of a closure operator over a finite ground set. By varying the prescribed facet $P_0$, we obtain all $2$-level polytopes in dimension $d$.
\end{abstract}


\section{Introduction}

A polytope $P \subseteq \R^d$ is said to be \emph{$2$-level} if every hyperplane $H$ that is facet-defining for $P$ has a parallel hyperplane $H'$ that contains all the vertices of $P$ which are not contained in $H$. In particular, if $P$ is empty or a point, it is $2$-level.

Some well known families of polytopes turn out to be $2$-level. For instance, cubes and cross-polytopes (more generally, Hanner polytopes~\cite{Hanner56}), Birkhoff polytopes~\cite{Birkhoff46} (more generally, polytopes of the form $P = \{x \in [0,1]^d \mid Ax = b\}$ where $A \in \Z^{m \times d}$ is totally unimodular and $b \in \Z^m$), order polytopes~\cite{Stanley86}, stable set polytopes of perfect graphs~\cite{Chvatal75} and their twisted prisms, the Hansen polytopes~\cite{Hansen77}, and spanning tree polytopes of series-parallel graphs~\cite{Grande16} all are $2$-level polytopes. Interestingly, it seems that there are only very few further examples of $2$-level polytopes known beyond this short list. This is in constrast with the fact that $2$-level polytopes appear, as we show next, in different areas of mathematics and theoretical computer science.

\subsection*{Motivations}

Let $V \subseteq \R^d$ be a finite set and $k$ be a positive integer. A polynomial $f(x) \in \R_{\leqslant 1}[x]$ of degree at most~$1$ is said to be \emph{$(1,k)$-SOS} on $V$ if there exist polynomials $g_1(x), \ldots, g_n(x) \in \R_{\leqslant k}[x]$ of degree at most~$k$ such that 
\[
f(x) = \sum_{i=1}^n g^2_i(x) \quad \text{for every } x \in V\,.
\]
The \emph{$k$-th theta body} of $V$ is the convex relaxation of the convex hull of $V$ defined by the linear inequalities $f(x) \geqslant 0$ where $f(x)$ is $(1,k)$-SOS on $V$. The \emph{theta rank} of $V$ is defined as the smallest $k$ such that this relaxation is exact, that is, the smallest $k$ such that for every valid linear inequality $f(x) \geqslant 0$, the affine form $f(x)$ is $(1,k)$-SOS on $V$. These notions were introduced by Gouveia, Parrilo and Thomas~\cite{Gouveia10}. Answering a question of Lov\'asz~\cite{Lovasz03}, they proved that a finite set has theta rank~$1$ if and only if it is the vertex set of a $2$-level polytope.

By virtue of this result and of the connection between sum-of-squares and semidefinite programming (see, e.g., \cite{Blekherman12} for more details), $2$-level polytopes are particularly well behaved from the point of view of optimization: any linear optimization problem over a $2$-level polytope in $\R^d$ can be reformulated as a semidefinite programming problem over $(d+1) \times (d+1)$ symmetric matrices. More precisely, it is known that $2$-level polytopes have minimum positive semidefinite rank (or positive semidefinite extension complexity) among all polytopes of the same dimension. In other words, $2$-level $d$-polytopes have positive semidefinite rank equal to $d+1$~\cite{Gouveia13}. 
For example, stable set polytopes of perfect graphs are one of the most prominent examples of $2$-level polytopes. To our knowledge, the fact that these polytopes have small positive semidefinite rank is the only known reason why one can efficiently find a maximum stable set in a perfect graph~\cite{Grotschel93}.

Moreover, $2$-level polytopes are also of interest in communication complexity since they provide interesting instances to test the \emph{log-rank conjecture}~\cite{Lovasz88}, one of the fundamental open problems in the area. This conjecture asserts that every $0/1$-matrix $M$ can be computed by a deterministic communication protocol of complexity at most $\polylog(\rk(M))$, which implies that the nonnegative rank of every $0/1$-matrix $M$ is at most $2^{\polylog(\rk(M))}$. Since every $d$-dimensional 2-level polytope has a slack matrix that is a $0/1$-matrix of rank $d+1$ (see Section~\ref{sec:slack-matrix}), the log-rank conjecture implies that every $2$-level $d$-polytope has nonnegative rank (or linear extension complexity) at most $2^{\polylog(d)}$. This is known for stable set polytopes of perfect graphs~\cite{Yannakakis91}, but appears to be open for general $2$-level polytopes.

There are more reasons to study $2$-level polytopes beyond those given above, in particular, in the context of volume computation and Erhart theory in which $2$-level polytopes originally appeared, see, e.g,~\cite{Stanley80}, and in statistics~\cite{Sullivant06}.

\subsection*{Contribution and outline}

In this paper we study the problem of enumerating all combinatorial types of $2$-level polytopes of a fixed dimension~$d$. This is equivalent to enumerating all $2$-level polytopes, up to \emph{affine} equivalence, see Lemma \ref{lem:2L-isomorphism}. For a definition of affine and combinatorial equivalence, see~\cite[Chapter 0]{Ziegler95}. Since every $2$-level polytope is affinely equivalent to a 0/1-polytope, one might think to compute all $2$-level polytopes of a given dimension simply by enumerating all 0/1-polytopes of that dimension and discarding the polytopes which are not $2$-level.
However, the complete enumeration of $d$-dimensional 0/1-polytopes has been implemented only for $d\leqslant 5$~\cite{Aichholzer00}. The author of the same paper has enumerated all $6$-dimensional 0/1-polytopes having up to 12 vertices, but the complete enumeration even for this low dimension is not expected to be feasible: the output of the combinatorial types alone is so huge that it is not currently possible to store it or search it efficiently~\cite{Ziegler00}. Thus for all but the lowest dimensions, there is no hope of working with a pre-existing list of 0/1-polytopes, and it is desirable to find an efficient algorithm which computes $2$-level polytopes from scratch. We present the first algorithm to enumerate all combinatorial types of $2$-level polytopes of a given dimension $d$. The algorithm uses new structural results on $2$-level polytopes which we develop here.

Our starting point is a pair of full-dimensional embeddings of a given $2$-level $d$-polytope defined in Section~\ref{sec:embeddings}. In one embedding, which we refer to as the $\mathcal H$-embedding, the facets have $0/1$-coefficients. In the other ---the $\mathcal V$-embedding--- the vertices have $0/1$-coordinates. The $\mathcal H$- and $\mathcal V$-embeddings are determined and linked by a structure, which we call a simplicial core (see Section~\ref{sec:simplicial-cores}).

We describe the enumeration algorithm in detail in Sections~\ref{sec:enum-algorithm}, \ref{sec:closure} and \ref{sec:reduction}. It computes a complete list $L_d$ of non-isomorphic $2$-level $d$-polytopes using the list $L_{d-1}$ of $2$-level $(d-1)$-polytopes. The algorithm is based on the fact that $L_d$ is the union of $L_d(P_0)$ for $P_0\in L_{d-1}$, where $L_d(P_0)$ is the collection of all 2-level $d$-polytopes that have $P_0$ as a facet. Indeed, every facet of a $2$-level polytope is $2$-level (see Lemma~\ref{lem:faces} below) and thus the above union equals $L_d$. Our enumeration strategy is illustrated in Figure~\ref{fig:2L_graph}.

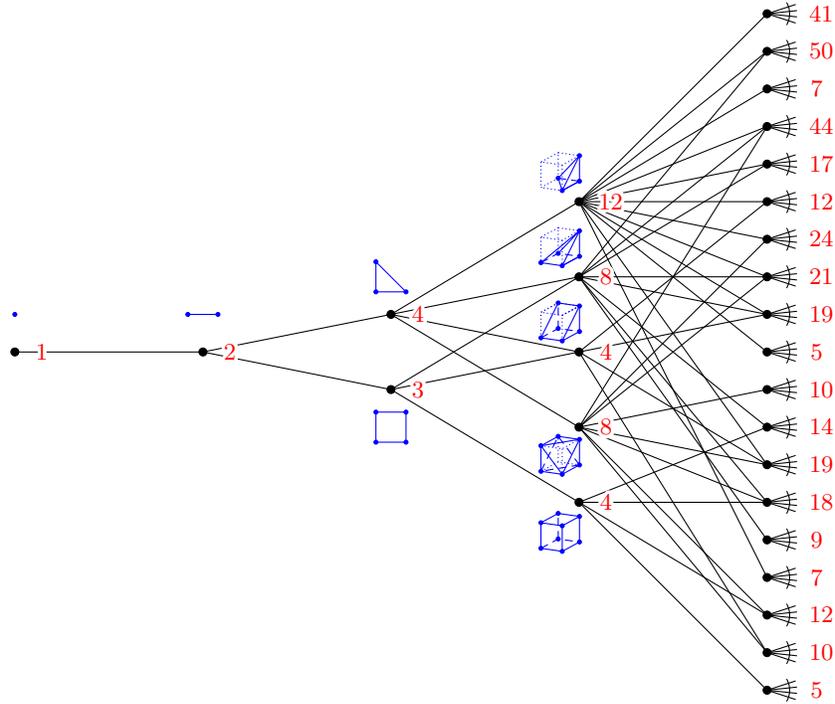
\begin{figure}[t!]\centering
\begin{tikzpicture}[line join=round,line width=.4pt]
\def\w{2.5}

\draw (0,0)--(1*\w,0);

\draw (1*\w,0)--(2*\w,.5);
\draw (1*\w,0)--(2*\w,-.5);

\draw (2*\w,.5)--(3*\w,2);
\draw (2*\w,.5)--(3*\w,1);
\draw (2*\w,.5)--(3*\w,0);
\draw (2*\w,.5)--(3*\w,-1);

\draw (2*\w,-.5)--(3*\w,1);
\draw (2*\w,-.5)--(3*\w,0);
\draw (2*\w,-.5)--(3*\w,-2);

\foreach \i in {1,...,10,13,16}{
\draw (3*\w,2)--(4*\w,5-.5*\i);}

\foreach \i in {2,4,5,8,9,12,14,15}{
\draw (3*\w,1)--(4*\w,5-.5*\i);}

\foreach \i in {6,9,13,18}{
\draw (3*\w,0)--(4*\w,5-.5*\i);}

\foreach \i in {4,7,8,11,13,14,17,18}{
\draw (3*\w,-1)--(4*\w,5-.5*\i);}

\foreach \i in {12,14,17,19}{
\draw (3*\w,-2)--(4*\w,5-.5*\i);}

{\tikzstyle{every node}=[text= red,fill=black,draw=black,circle,inner sep=0pt,minimum size=3pt]
\node[label={[xshift=2mm]right:\contour{white}{\color{red}1}}] at (0,0){};
\node[label={[xshift=2mm]right:\contour{white}{\color{red}2}}] at (1*\w,0){};

\node[label={[xshift=2mm]right:\contour{white}{\color{red}4}}] at (2*\w,.5){};
\node[label={[xshift=2mm]right:\contour{white}{\color{red}3}}] at (2*\w,-.5){};

\node[label={[xshift=2mm]right:\contour{white}{\color{red}12}}] at (3*\w,2){};
\node[label={[xshift=2mm]right:\contour{white}{\color{red}8}}] at (3*\w,1){};
\node[label={[xshift=2mm]right:\contour{white}{\color{red}4}}] at (3*\w,0){};
\node[label={[xshift=2mm]right:\contour{white}{\color{red}8}}] at (3*\w,-1){};
\node[label={[xshift=2mm]right:\contour{white}{\color{red}4}}] at (3*\w,-2){};

\node[text=red,label={[xshift=5mm]right:41}] at (4*\w,5-.5*1){};
\node[text=red,label={[xshift=5mm]right:50}] at (4*\w,5-.5*2){};
\node[text=red,label={[xshift=5mm]right:7}] at (4*\w,5-.5*3){};
\node[text=red,label={[xshift=5mm]right:44}] at (4*\w,5-.5*4){};
\node[text=red,label={[xshift=5mm]right:17}] at (4*\w,5-.5*5){};
\node[text=red,label={[xshift=5mm]right:12}] at (4*\w,5-.5*6){};
\node[text=red,label={[xshift=5mm]right:24}] at (4*\w,5-.5*7){};
\node[text=red,label={[xshift=5mm]right:21}] at (4*\w,5-.5*8){};
\node[text=red,label={[xshift=5mm]right:19}] at (4*\w,5-.5*9){};
\node[text=red,label={[xshift=5mm]right:5}] at (4*\w,5-.5*10){};
\node[text=red,label={[xshift=5mm]right:10}] at (4*\w,5-.5*11){};
\node[text=red,label={[xshift=5mm]right:14}] at (4*\w,5-.5*12){};
\node[text=red,label={[xshift=5mm]right:19}] at (4*\w,5-.5*13){};
\node[text=red,label={[xshift=5mm]right:18}] at (4*\w,5-.5*14){};
\node[text=red,label={[xshift=5mm]right:9}] at (4*\w,5-.5*15){};
\node[text=red,label={[xshift=5mm]right:7}] at (4*\w,5-.5*16){};
\node[text=red,label={[xshift=5mm]right:12}] at (4*\w,5-.5*17){};
\node[text=red,label={[xshift=5mm]right:10}] at (4*\w,5-.5*18){};
\node[text=red,label={[xshift=5mm]right:5}] at (4*\w,5-.5*19){};
}

\foreach \i in {1,...,19}{
\begin{scope}[shift={(4*\w,5-.5*\i)}]
\draw (0,0) -- (18:.4cm);
\draw (0,0) -- (6:.4cm);
\draw (0,0) -- (-6:.4cm);
\draw (0,0) -- (-18:.4cm);
\draw (0,0) ++(30:.3) arc (30:-30:.3);
\end{scope}
}

\begin{scope}[shift={(0,0+.5)},scale=.2]
\node[fill=blue,draw=blue,circle,inner sep=0pt,minimum size=1.5pt] at(0,0){};
\end{scope}

\begin{scope}[shift={(1*\w,0+.5)},scale=.2]
{\tikzstyle{every node}=[fill=blue,draw=blue,circle,inner sep=0pt,minimum size=1.5pt]
\node at(-1,0){};
\node at (1,0){};}
\draw[blue](-1,0)--(1,0);
\end{scope}

\begin{scope}[shift={(2*\w,.5+.5)},scale=.2]
{\tikzstyle{every node}=[fill=blue,draw=blue,circle,inner sep=0pt,minimum size=1.5pt]
\node at(-1,-1){};
\node at (1,-1){};
\node at (-1,1){};}
\draw[blue](-1,1)--(-1,-1)--(1,-1)--cycle;
\end{scope}

\begin{scope}[shift={(2*\w,-.5-.5)},scale=.2]
{\tikzstyle{every node}=[fill=blue,draw=blue,circle,inner sep=0pt,minimum size=1.5pt]
\node at(-1,-1){};
\node at (1,-1){};
\node at (-1,1){};
\node at (1,1){};}
\draw[blue](-1,1)--(-1,-1)--(1,-1)--(1,1)--cycle;
\end{scope}

\begin{scope}[shift={(2.9*\w,2+.4)},scale=.17,rotate around y=-18]

\tikzset{frame/.style={draw=blue,densely dotted},
hidden/.style={draw=blue,densely dashed}}
\draw[frame](-1,1,-1)--(1,1,-1)--(1,1,1)--(-1,1,1)--cycle;

\draw[frame](-1,-1,-1)--(-1,-1,1)--(1,-1,1);
\draw[frame](1,1,1)--(1,-1,1);
\draw[frame](-1,1,1)--(-1,-1,1);
\draw[frame](-1,-1,-1)--(-1,1,-1);

\draw[hidden](-1,-1,-1)--(1,-1,-1);
\draw[blue](-1,-1,-1)--(1,-1,1)--(1,-1,-1);
\draw[blue](1,1,-1)--(1,-1,1)--(1,-1,-1)--cycle;
\draw[blue](-1,-1,-1,)--(1,1,-1);

{\tikzstyle{every node}=[fill=blue,draw=blue,circle,inner sep=0pt,minimum size=1.5pt]
\node at(-1,-1,-1){};
\node at (1,-1,1){};
\node at (1,-1,-1){};
\node at (1,1,-1){};}	
\end{scope}

\begin{scope}[shift={(2.9*\w,1+.4)},scale=.17,rotate around y=-18]
\tikzset{frame/.style={draw=blue,densely dotted},
hidden/.style={draw=blue,densely dashed}}
\draw[frame](-1,1,-1)--(1,1,-1)--(1,1,1)--(-1,1,1)--cycle;

\draw[frame](1,1,1)--(1,-1,1);
\draw[frame](-1,1,1)--(-1,-1,1);
\draw[frame](-1,-1,-1)--(-1,1,-1);

\draw[hidden](-1,-1,1)--(-1,-1,-1)--(1,-1,-1);
\draw[blue](-1,-1,1)--(1,-1,1)--(1,-1,-1);
\draw[blue](1,1,-1)--(1,-1,1)--(1,-1,-1)--cycle;
\draw[blue](-1,-1,-1,)--(1,1,-1)--(-1,-1,1);

{\tikzstyle{every node}=[fill=blue,draw=blue,circle,inner sep=0pt,minimum size=1.5pt]
\node at(-1,-1,-1){};
\node at (1,-1,1){};
\node at (1,-1,-1){};
\node at (1,1,-1){};
\node at (-1,-1,1){};}	
\end{scope}

\begin{scope}[shift={(2.9*\w,0+.4)},scale=.17,rotate around y=-18]
\tikzset{frame/.style={draw=blue,densely dotted},
hidden/.style={draw=blue,densely dashed}}
\draw[frame](-1,1,-1)--(1,1,-1)--(1,1,1)--(-1,1,1)--cycle;

\draw[frame](1,1,-1)--(1,1,1)--(1,-1,1);
\draw[frame](1,1,-1)--(1,1,1);
\draw[frame](-1,1,-1)--(-1,1,1)--(-1,-1,1);

\draw[hidden](-1,-1,1)--(-1,-1,-1)--(1,-1,-1);
\draw[hidden](-1,-1,-1)--(-1,1,-1);

\draw[blue](-1,1,-1)--(1,1,-1)--(1,-1,1)--(-1,-1,1)--cycle;
\draw[blue](1,-1,1)--(1,-1,-1)--(1,1,-1);

{\tikzstyle{every node}=[fill=blue,draw=blue,circle,inner sep=0pt,minimum size=1.5pt]
\node at(-1,-1,-1){};
\node at(-1,1,-1){};
\node at (1,-1,1){};
\node at (1,-1,-1){};
\node at (1,1,-1){};
\node at (-1,-1,1){};}	
\end{scope}

\begin{scope}[shift={(2.9*\w,-1-.375)},scale=.17,rotate around y=-18]
\tikzset{frame/.style={draw=blue,densely dotted},
hidden/.style={draw=blue,densely dashed}}

\draw[frame](-1,-1,1)--(-1,-1,-1)--(1,-1,-1);
\draw[frame](-1,-1,-1)--(-1,1,-1);

\draw[frame](-1,1,1)--(1,1,1)--(1,1,-1);
\draw[frame](1,1,1)--(1,-1,1);

\draw[hidden](1,-1,-1)--(-1,-1,1)--(-1,1,-1)--cycle;

\draw[blue](-1,1,1)--(1,1,-1)--(1,-1,1)--cycle;
\draw[blue](1,-1,1)--(1,-1,-1)--(1,1,-1);
\draw[blue](1,-1,1)--(-1,-1,1)--(-1,1,1);
\draw[blue](-1,1,1)--(-1,1,-1)--(1,1,-1);

{\tikzstyle{every node}=[fill=blue,draw=blue,circle,inner sep=0pt,minimum size=1.5pt]
\node at(-1,1,-1){};
\node at (1,-1,1){};
\node at (1,-1,-1){};
\node at (1,1,-1){};
\node at (-1,1,1){};
\node at (-1,-1,1){};}	
\end{scope}

\begin{scope}[shift={(2.9*\w,-2-.4)},scale=.17,rotate around y=-18]
\tikzset{hidden/.style={draw=blue,densely dashed}}

\draw[hidden](-1,-1,1)--(-1,-1,-1)--(1,-1,-1);
\draw[hidden](-1,-1,-1)--(-1,1,-1);

\draw[blue](-1,1,-1)--(1,1,-1)--(1,1,1)--(-1,1,1)--cycle;
\draw[blue](-1,1,1)--(-1,-1,1)--(1,-1,1)--(1,-1,-1)--(1,1,-1);
\draw[blue](1,1,1)--(1,-1,1);

{\tikzstyle{every node}=[fill=blue,draw=blue,circle,inner sep=0pt,minimum size=1.5pt]
\node at(-1,-1,-1){};
\node at(-1,1,-1){};
\node at (1,-1,1){};
\node at (1,-1,-1){};
\node at (1,1,-1){};
\node at (-1,-1,1){};
\node at(-1,1,1){};
\node at(1,1,1){};}	
\end{scope}

\end{tikzpicture}

\caption{(Rotated) Hasse diagram of the poset of combinatorial types of 2-level polytopes with respect to inclusion. In the figure, an edge between the combinatorial types of the polytopes $P$ and $F$ indicates that $P$ has a facet that is isomorphic to $F$. Combinatorial types of a fixed dimension are sorted top to bottom lexicographically by their $f$-vector. Thus the first type is always that of the simplex. Labels on the nodes of the diagram are the number of times a given combinatorial type appears as a facet of another type.\label{fig:2L_graph}}
\end{figure}

For every polytope $P_0 \in L_{d-1}$, we perform the following steps. First, we embed $P_0$ in the hyperplane $\{x \in \R^d \,|\, x_1 = 0\} \simeq \R^{d-1}$ (using a $\mathcal{H}$-embedding). Then, we compute a collection $\mathcal{A}$ of point sets $A\subseteq \{x \in \R^d \,|\, x_1=1\}$ such that for each $2$-level polytope $P \in L_d(P_0)$, there exists $A \in \mathcal{A}$ with $P$ isomorphic to $\conv(P_0 \cup A)$. For each $A \in \mathcal{A}$, we compute the slack matrix of $Q \coloneqq \conv (P_0 \cup A)$ and add $Q$ to the list $L_d$, provided that it is $2$-level and not isomorphic to any of the polytopes already generated by the algorithm.

The efficiency of this approach depends greatly on how the collection $\mathcal{A}$ is chosen. Here, we exploit the $\mathcal H$- and $\mathcal V$-embeddings to define a proxy for the notion of $2$-level polytopes in terms of closed sets with respect to a certain closure operator and use this proxy to construct a suitable collection $\mathcal{A}$. Moreover, we develop tools to compute convex hulls combinatorially and avoid resorting to standard convex hull algorithms (see Section \ref{sec:impl}), providing a significative speedup in the computations.

We implement this algorithm and run it to obtain $L_d$ for $d \leqslant 7$. The outcome of our experiments is discussed in Section~\ref{sec:exp_results}. In particular, our results show that the number of combinatorial types of $2$-level $d$-polytopes is surprisingly small for low dimensions $d$. We conclude the paper by discussing research questions inspired by our experiments, see Section \ref{sec:final_remarks}. 

\subsection*{Previous related work}

The enumeration of all combinatorial types of point configurations and polytopes is a fundamental problem in discrete and computational geometry. Latest results in~\cite{Fukuda13} report complete enumeration of polytopes for dimension $d=3,4$ with up to $8$ vertices and $d=5,6,7$ with up to $9$ vertices. For $0/1$-polytopes this is done completely for $d \leqslant 5$ and $d=6$ with up to $12$ vertices~\cite{Aichholzer00}. In our approach, we use techniques from formal concept analysis, in particular we use a previously existing algorithm to enumerate all concepts of a relation, see~\cite{Ganter91,Kuznetsov02}.

Regarding 2-level polytopes, Grande and Sanyal~\cite{Grande16} give an excluded minor
characterization of $2$-level matroid base polytopes. Grande and Ru{\'e}~\cite{Grande15b} 
give a $O(c^d)$ lower bound on the number of 2-level matroid $d$-polytopes. Finally, 
Gouveia \emph{et al.}~\cite{Gouveia15} give a complete classification of polytopes with 
minimum positive semidefinite rank, which generalize $2$-level polytopes, in dimension $d=4$.

\subsection*{Conference versions}

A first version of the enumeration algorithm together with the experimental results for $d \leqslant 6$ appeared in~\cite{Bohn15}. An improvement of the algorithm that yielded enumeration results in $d=7$ appeared as part of~\cite{Fiorini16}. We point out that this 
paper is missing one $2$-level polytope for $d = 7$, see~\cite[Table 2]{Fiorini16}. The correct number of $2$-level $7$-polytopes is provided here, see Table~\ref{tbl:2Lnums} below. Besides this correction, the present paper contains a full correctness proof for the enumeration algorithm. Moreover, compared to~\cite{Fiorini16}, the algorithm was further optimized. The two main differences are: the more drastic reductions we perform on the ground set, and the fact that we bypass convex hull computations completely. These are replaced by a combinatorial polytope verification procedure. More details can be found in Section~\ref{sec:impl}.


\section{Embeddings} \label{sec:embeddings}

In this section, after fixing some notation, we discuss the notion of simplicial core. This is then used to define the two types of embeddings that we use for $2$-level polytopes. Finally, we establish important properties of these embeddings that are used later in the enumeration algorithm.

\subsection{Notations}

For basic notions on polytopes that do not appear here, we refer the reader to \cite{Ziegler95}. We use $\vertex(P)$ to denote the vertex set of a polytope $P$. Let $d$ denote a positive integer, which we use most of the time to denote the dimension of the ambient space. A \emph{$d$-polytope} is simply a polytope of dimension $d$. We let $[d] \coloneqq \{1,\dots,d\}$. For $x \in \R^d$ and $E \subseteq [d]$, we let $x(E):=\sum_{i \in E} x_i$.

\subsection{Simplicial cores} \label{sec:simplicial-cores}

We introduce the structural notion of simplicial core of a polytope, which will be used in the enumeration algorithm to ease the counting of combinatorial types of $2$-level polytopes.

\begin{defn}[Simplicial core] \label{def:simplicial-core}
A \emph{simplicial core} for a $d$-polytope $P$ is a $(2d+2)$-tuple $(F_1,\ldots,F_{d+1};$ $v_1,\ldots,v_{d+1})$ of facets and vertices of $P$ such that each facet $F_i$ does not contain vertex $v_i$ but contains vertices $v_{i+1}$, \ldots, $v_{d+1}$.
\end{defn}

\begin{figure}[ht!]\centering
\begin{tikzpicture}[line join=round]
\tikzset{f-core/.style={line width=1.6pt,draw,fill=lightgray,fill opacity =.5},
vertex/.style={fill,draw,circle,inner sep=1.7pt,minimum width=1pt}}


\draw[f-core] (0,0,0)-- (0,0,2) -- (2,0,2) -- (2,0,0) -- cycle;
\draw[f-core] (0,0,0)-- (0,2,0) -- (2,2,0) -- (2,0,0) -- cycle;
\draw[f-core] (0,0,0)-- (0,2,0) -- (0,2,2) -- (0,0,2) -- cycle;
\draw (0,2,2)-- (2,2,2);
\draw[f-core] (2,0,0)-- (2,2,0) -- (2,2,2) -- (2,0,2) -- cycle;

\node[vertex,fill=white] at (2,2,0) {};
\node[vertex,fill=white] at (0,2,2) {};
\node[vertex,fill=white] at (2,0,2) {};
\node[vertex,fill=white] at (2,2,2) {};

\node[vertex,label={right:$v_4$}] at (2,0,0) {};
\node[vertex,label={left:$v_2$}] at (0,0,2) {};
\node[vertex,label={above:$v_1$}] at (0,2,0) {};
\node[vertex,label={left:$v_3$}] at (0,0,0) {};

\node at (1,0,1) {$F_1$};
\node[below left] at (1,1,0) {$F_2$};
\node at (2,1,1) {$F_3$};
\node at (0,1,1) {$F_4$};

\end{tikzpicture}
\caption{Some simplicial core $(F_1,F_2,F_3,F_4;v_1,v_2,v_3,v_4)$ for the $3$-dimensional cube.}
\end{figure}
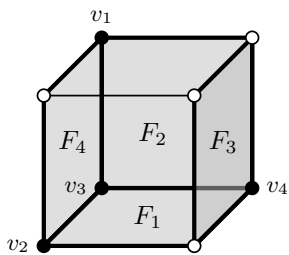
The concept of simplicial core appeared in relation with $2$-level polytopes  already in \cite{Stanley80} and with polytopes of minimum nonnegative rank in \cite{GouvRobThom13}. 

The following lemma proves the existence of simplicial cores. Although this is known (see \cite[Proposition 3.2]{GouvRobThom13}), we provide a proof for completeness.

\begin{lem} \label{lem:simplicial_cores}
For every $d$-polytope $P$ there exist facets $F_1$, \ldots, $F_{d+1}$ and vertices $v_1$, \ldots, $v_{d+1}$ of $P$ such that $(F_1,\ldots,F_{d+1};v_1,\ldots,v_{d+1})$ is a simplicial core for $P$.
\end{lem}
\begin{proof}
The proof is by induction on the dimension. For a $1$-polytope $P:=\conv(\{v_1,v_2\})$ we can take $F_1 \coloneqq \{v_2\}$ and $F_2 \coloneqq \{v_1\}$. For the induction step, let $P'$ be a facet of $P$. Thus $P'$ is a $(d-1)$-polytope. By the induction hypothesis, there are facets $F'_2, \ldots, F'_{d+1}$ and vertices $v_2,\ldots, v_{d+1}$ of $P'$ such that $(F'_2, \ldots, F'_{d+1};v_2,\ldots,v_{d+1})$ is a simplicial core for $P'$. Now let $F_1 \coloneqq P'$ and, for $i$, $2 \leqslant i \leqslant d+1$, let $F_i$ be the unique facet of $P$ that contains $F'_i$ and is distinct from $F_1$. Let $v_1$ be any vertex of $P$ which does not belong to the facet $F_1$. 

By construction, each $F_i$ contains $v_{i+1}, \ldots, v_{d+1}$. Moreover, $F_1$ does not contain $v_1$ and, for $i \geq 2$, facet $F_i$ does not contain $v_i$ because otherwise $F'_i = F_1 \cap F_i$ would contain $v_i$, contradicting our hypothesis.
\end{proof}

We point out that the proof of Lemma~\ref{lem:simplicial_cores} gives us more than the existence of simplicial cores: for every polytope $P$, every facet $F_1$ of $P$ and every simplicial core $\Gamma' \coloneqq (F'_2, \ldots, F'_{d+1};v'_2,\ldots,v'_{d+1})$ of $F_1$, there exists a simplicial core $\Gamma \coloneqq (F_1,F_2,\ldots,F_{d+1};v_1,v_2,\ldots,v_{d+1})$ of $P$ \emph{extending} $\Gamma'$ in the sense that $F'_i = F_1 \cap F_i$ and $v'_i = v_i$ for $i \geqslant 2$. 

Finally, we make the following observation. Let $(F_1,\ldots,F_{d+1};v_1,\ldots,v_{d+1})$ be a simplicial core of some polytope $P$. For each $i$, the affine hull of $F_i$ contains $v_j$ for $j > i$, but does not contain $v_i$. Therefore, the vertices of a simplicial core are affinely independent. That is, $v_1$, \ldots, $v_{d+1}$ form the vertex set of a $d$-simplex contained in $P$.

\subsection{Slack matrices and slack embeddings} \label{sec:slack-matrix}

\begin{defn}[Slack matrix \cite{Yannakakis91}]
The \emph{slack matrix} of a polytope $P \subseteq \R^{d}$ with $m$ facets $F_1$, \ldots, $F_m$ and $n$ vertices $v_1$, \ldots, $v_n$ is the $m \times n$ nonnegative matrix $S = S(P)$ such that $S_{ij}$ is the \emph{slack} of the vertex $v_j$ with respect to the facet $F_i$, that is, $S_{ij} = g_i(v_j)$ where $g_i : \R^{d} \to \R$ is any affine form such that $g_i(x) \geqslant 0$ is valid for $P$ and $F_i = \{x \in P \mid g_i(x) = 0\}$.
\end{defn}

The slack matrix of a polytope is defined up to scaling its rows by positive reals.
Notice that simplicial cores for $P$ correspond to $(d+1) \times (d+1)$ submatrices of $S(P)$ that are invertible and lower-triangular, for some ordering of rows and columns.

The slack matrix provides a canonical way to embed any polytope, which we call the \emph{slack embedding}. This embedding maps every vertex $v_j$ to the corresponding column $S^j \in \R_+^m$ of the slack matrix $S = S(P)$. The next lemma shows that every polytope is affinely isomorphic to the convex hull of the columns of its slack matrix.

\begin{lem} \label{lem:full_slack_embed}
Let $P$ be a $d$-polytope having facet-defining inequalities $g_1(x) \geqslant 0$, \ldots, $g_m(x) \geqslant 0$, and vertices $v_1$, \ldots, $v_n$. If $\sigma$ denotes a map from the affine hull $\aff(P)$ of $P$ to $\R^m$ defined by $\sigma(x)_i \coloneqq g_i(x)$ for all $x \in \aff(P)$, then the polytopes $P$ and $\sigma(P)$ are affinely equivalent.
\end{lem}
\begin{proof}
The map $\sigma : \aff(P) \to \R^m$ is affine, and injective because it maps the vertices of any simplicial core for $P$ to affinely independent points. The result follows.
\end{proof}

By definition, a polytope $P$ is $2$-level if and only if $S(P)$ can be scaled to be a 0/1 matrix. Given a $2$-level polytope, we henceforth always assume that its facet-defining inequalities are scaled so that the slacks are 0/1. Thus, the slack embedding of a $2$-level polytope depends only on the \emph{support}\footnote{The \emph{support} of a matrix $M$ is the collection of all pairs $(i,j)$ such that $M_{i,j} \neq 0$.} of its slack matrix, which only depends on its combinatorial structure. As a consequence, we have the following result: 

\begin{lem} \label{lem:2L-isomorphism}
Two $2$-level polytopes are affinely equivalent if and only if they have the same combinatorial type.
\end{lem}
\begin{proof}
The result follows directly from Lemma~\ref{lem:full_slack_embed} and the fact that, since 
the slack matrix of a $2$-level polytope $P$ is the facet vs.\ vertex non-incidence matrix of $P$, it depends only on the combinatorial type of $P$.
\end{proof}

For the sake of completeness, we state the following basic result about $2$-level polytopes, which easily follows from our discussion of slack matrices (see also~\cite[Corollary 4.5 (2)]{Gouveia10}).

\begin{lem}\label{lem:faces}
Each face of a $2$-level polytope is a $2$-level polytope.
\end{lem}
\begin{proof}
Let $S = S(P) \in \{0,1\}^{m \times n}$ be the slack matrix of some $2$-level polytope $P$, and let $F$ be any face of $P$. If $F$ is empty or a point, then $F$ is $2$-level by definition. Otherwise,  the slack matrix of $F$ is a submatrix of $S$, which implies that $F$ is $2$-level.
\end{proof}

\subsection{$\mathcal H$- and $\mathcal V$-embeddings} \label{sec:VH-embeddings}
Although canonical, the slack embedding is never full-dimensional, which can be a disadvantage. To remedy this, we use simplicial cores to define two types of embeddings that are full-dimensional. 
Let $P$ be a $2$-level $d$-polytope with $m$ facets and $n$ vertices, and let $\Gamma \coloneqq (F_1,\ldots,F_{d+1};v_1,\ldots,v_{d+1})$ be a simplicial core for $P$. Since $v_1$,\ldots, $v_{d+1}$ are affinely independent, the images of $v_1$,\ldots, $v_{d+1}$ uniquely define an affine embedding of $P$.

The slack matrix $S(P)$ is a $0/1$ matrix. Moreover, we assume that the rows and columns of the slack matrix $S(P)$ are ordered compatibly with the simplicial core $\Gamma$, so that the $i$-th row of $S(P)$ corresponds to facet $F_i$ for $1 \leqslant i \leqslant d+1$ and the $j$-th column of $S(P)$ corresponds to vertex $v_j$ for $1 \leqslant j \leqslant d+1$.

\begin{defn}[$\mathcal{H}$-embedding]\label{def:embedding}
The \emph{$\mathcal{H}$-embedding} of $P$ with respect to the simplicial core $\Gamma \coloneqq (F_1,\ldots,F_{d+1};v_1,\ldots,v_{d+1})$  is defined by mapping $v_j$, $1 \leqslant j \leqslant d$ to the unit vector $e_j$ of $\R^d$ and mapping $v_{d+1}$ to the origin. 
\end{defn}

\begin{defn}[$\mathcal{V}$-embedding]
Let $S \coloneqq S(P)$ be the slack matrix of $P$. The \emph{$\mathcal V$-embedding} of $P$ with respect to the simplicial core $\Gamma \coloneqq (F_1,\ldots,F_{d+1};v_1,\ldots,v_{d+1})$ is defined by mapping $v_j$, $1 \leqslant j \leqslant d+1$  to the point of $\R^d$ whose $i$-th coordinate is $S_{ij}$ for  $1 \leqslant i \leqslant d$. 
\end{defn}

Equivalently, the $\mathcal V$-embedding can be defined by the mapping
\(
x \mapsto Mx,
\)
where $M = M(\Gamma)$ is the top left $d \times d$ submatrix of $S(P)$ and $x \in \R^d$ is a point in the $\mathcal H$-embedding. In fact, the matrix $M$ maps the vertices of the simplicial core to the columns of the slack-matrix of $P$, that are precisely the vertices of $P$ in the $\mathcal V$-embedding.
The next lemma provides the main properties of these embeddings.

\begin{lem} \label{lem:embeddings}
Let $P$ be a $2$-level $d$-polytope and $\Gamma \coloneqq (F_1,\ldots,F_{d+1};v_1,\ldots,v_{d+1})$ be some simplicial core for $P$. Then the following properties hold:

\begin{itemize}
\item in the  $\mathcal H$-embedding of $P$ with respect to $\Gamma$, all the facets are of the form $x(E) \leqslant 1$ or $x(E) \geqslant 0$ for some nonempty $E \subseteq [d]$. 

\item in the $\mathcal V$-embedding of $P$ with respect to $\Gamma$, the $i$-th coordinate of a vertex is the slack with respect to facet $F_i$ of the corresponding vertex in $P$. In particular, in the $\mathcal V$-embedding all the vertices have $0/1$-coordinates.
\end{itemize}
\end{lem}

\begin{proof}
Let $g(x):=a_0 + \sum_{i=1}^d a_i x_i \geqslant 0$ be a facet defining inequality in the  $\mathcal H$-embedding. Since $P$ is $2$-level, we may assume that $g(x)$ takes $0/1$ values on the vertices of the $\mathcal H$-embedding. That is, on $e_j$, $1\leqslant j \leqslant d$ and the origin. Thus, $g(x)$ has either the form $\sum_{i \in E} x_i$ or the form $1 - \sum_{i \in E} x_i$ for some nonempty $E \subseteq [d]$.

Consider the $\mathcal V$-embedding with respect to $\Gamma$ and fix $i$, $1 \leqslant i \leqslant d$ arbitrarily. The $i$-th coordinate of a point in the $\mathcal V$-embedding and the map computing the slack with respect to the facet $F_i$ are two affine forms on $\aff(P)$, and their value coincide on the vertices $v_1, \ldots, v_{d+1}\in P$. The statement follows, since the affine hull of $v_1, \ldots, v_{d+1}$ equals $\aff(P)$.
\end{proof}

We use the notation $M = M(\Gamma)$ for the top $d \times d$ submatrix of $S(P)$, and we call this submatrix the \emph{embedding transformation matrix} of $\Gamma$. Note that every embedding transformation matrix $M$ is unimodular. Indeed, $M$ is an invertible, lower-triangular, $0/1$-matrix. Thus $\det(M)=1$. 

Below, we use the shorthand $M \cdot X$ for the set $\{Mx \mid x \in X\}$ where $M$ is a $d \times d$ matrix and $X$ is a subset of $\R^d$.

\begin{cor}\label{lem:possible_vertices}
Let $P$ be the $\mathcal H$-embedding of a $2$-level $d$-polytope with respect to a simplicial core $\Gamma$. Then $\vertex(P)=P \cap M^{-1} \cdot \{0,1\}^d=P \cap \Z^d\subseteq \Z^d$, where $M = M(\Gamma)$ is the embedding transformation matrix of $\Gamma$.
\end{cor}

\begin{proof}
The transformation matrix $M$ is unimodular, thus it maps integer vectors to integer vectors. 
For this reason $\vertex(P) = M^{-1} \cdot (\vertex(M \cdot P)) = M^{-1} \cdot (M \cdot P \cap \{0,1\}^d)$  since $M \cdot P$ is a $0/1$-polytope. From this identity, $\vertex(P) = P \cap M^{-1} \cdot \{0,1\}$. In the same fashion, using again the unimodularity of $M$ and the fact that $M \cdot P$ is integral, $\vertex(P) = M^{-1} \cdot (M \cdot P \cap \Z^d) = P \cap M^{-1} \cdot \Z^d = P \cap \Z^d$.
\end{proof}

It follows from Lemma~\ref{lem:embeddings} that any $\mathcal H$-embedding of a $2$-level $d$-polytope is of the form $P(H) \coloneqq \{x \in \R^d\mid 0 \leqslant x(E) \leqslant 1 \text{ for each } E \in \mathcal{E} \}
$, for some hypergraph $H = (V,\mathcal{E})$, with $V=[d]$. Observe that $P(H)$ is 2-level if and only if it is integral. For every hyperedge $E\in \mathcal{E}$, we refer to a pair of inequalities $0 \leqslant x(E) \leqslant 1$ as a pair of \emph{hyperedge constraints}.

Finally, we prove a surprising structural result for $2$-level polytopes: the local information of having a simple vertex has a huge impact on the entire structure of the polytope since it forces the polytope to be isomorphic to the stable set polytope of a perfect graph. We use this later in Section~\ref{sec:exp_results} to recognize stable set polytopes of perfect graphs among $2$-level polytopes. 

\begin{lem}\label{lem:simple_vertex}
Let $P$ be some $2$-level polytope. If $P$ has a simple vertex, then it is isomorphic to the stable set polytope of a perfect graph. 
\end{lem}

\begin{proof}
Let $v$ be some simple vertex of $P$. Assuming that $P$ is $d$-dimensional, $v$ is contained in $d$ edges of $P$. We denote by $[v,v_i]$, $i \in [d]$, the $d$ edges containing $v$.
There exist facets $F_1$, \ldots, $F_{d+1}$ of $P$ such that 
$\Gamma := (F_1,\ldots,F_{d+1}; v_1,\ldots,v_d,v_{d+1}:=v)$ is a simplicial core. The first $d$ facets are determined by the choice of $v$ since they are all the facets of $P$ containing $v$, while $F_{d+1}$ is any facet not containing $v$.
By construction, $M(\Gamma) = I_d$.

Consider the $\mathcal{H}$-embedding of $P$ with respect to $\Gamma$, which has the form $P(H) = \{x \in \R^d\mid 0 \leqslant x(E) \leqslant 1 \text{ for each } E \in \mathcal{E} \}
$, for some hypergraph $H = (V,\mathcal{E})$ with $V=[d]$. Since $M(\Gamma)$ is the $d \times d$ identity matrix, the $\mathcal{H}$- and $\mathcal{V}$-embeddings with respect to $\Gamma$ coincide. For this reason, $P(H) = P_+(H)$, where $P_+(H) := P(H) \cap \R^d_+$. 

Since $P$ is $2$-level, $P(H) = P_+(H)$ is integral. By \cite[Theorem 3.4]{Cornuejols00}, we conclude that $P_+(H)$ is the stable set polytope of some perfect graph. The result follows.
\end{proof}


\section{The enumeration algorithm}
\label{sec:enum-algorithm}

In this section, we provide a high-level description of the enumeration algorithm. At this stage, we omit some details, which we postpone to Sections \ref{sec:closure} and \ref{sec:reduction}, and convey the main ideas.

Suppose that we are given a list $L_{d-1}$ of $2$-level polytopes of dimension $d-1$, each stored with some simplicial core. More precisely, each polytope in $L_{d-1}$ is represented by its slack matrix, with rows and columns ordered in such a way that the top left $d \times d$ submatrix corresponds to the chosen simplicial core. 

Now pick $P_0 \in L_{d-1}$, and let $\Gamma_0$ denote its stored simplicial core. Let $M_{d-1} \coloneqq M(\Gamma_0)$ be the corresponding embedding transformation matrix. Thus $M_{d-1}$ is the top left $(d-1) \times (d-1)$ submatrix of the slack matrix of $P_0$.

%

We would like to enumerate all $2$-level $d$-polytopes $P$ which have a facet isomorphic to $P_0$, together with a simplicial core $\Gamma$, up to isomorphism. From Lemma \ref{lem:faces} we know that, by varying $P_0 \in L_{d-1}$, we are going to enumerate all combinatorial types of $2$-level $d$-polytopes.

Consider a $2$-level $d$-polytope $P$ having a facet isomorphic to $P_0$. For the sake of simplicity, we assume that $P_0$ is actually a facet of $P$. Then we can find a simplicial core $\Gamma$ of $P$ that extends the chosen simplicial core $\Gamma_0$ of $P_0$, see the discussion after Lemma~\ref{lem:simplicial_cores}. In terms of slack matrices, this implies that the corresponding embedding transformation matrix $M_{d}:=M(\Gamma)$ takes the form
\begin{equation}\label{eq:M_d}
M_d = M_d(c) =
\begin{pmatrix}
1 & \begin{matrix} 0 & \cdots & 0 \end{matrix}\\
\begin{matrix} c_1 \\ \vdots\\ c_{d-1} \end{matrix} & \fbox{\parbox[c][1cm][c]{1cm}{\centering$M_{d-1}$}} \\
\end{pmatrix}\,.
\end{equation}

A priori, we do not know the vector $c \in \{0,1\}^{d-1}$. Suppose for now that we fix $c \in \{0,1\}^{d-1}$, so that $M_d$ is completely defined. From Corollary~\ref{lem:possible_vertices}, we know that in the $\mathcal{H}$-embedding of $P$ corresponding to $\Gamma$, we have $\vertex(P) \subseteq M_d^{-1} \cdot \{0,1\}^d$. By construction, $P_0$ is the first facet of $\Gamma$. Hence, $P_0$ is the facet of $P$ defined by $x_1 \geqslant 0$. Since $P$ is $2$-level, we can decompose its vertex set as $\vertex(P) = \vertex(P_0) \cup \vertex(P_1)$ where $P_1$ is the face of $P$ opposite to $P_0$, defined by $x_1 \leqslant 1$. Notice that $\vertex(P_1) \subseteq M_d^{-1} \cdot (\{1\} \times \{0,1\}^{d-1})$, and moreover $e_1 \in \vertex(P_1)$, since we are considering an $\mathcal{H}$-embedding of $P$. 

Now consider the set
\begin{equation}\label{eq:first_X}
\firstX = \firstX(P_0,\Gamma_0) \coloneqq \bigcup_{c \in \{0,1\}^{d-1}} M_d(c)^{-1} \cdot (\{1\} \times \{0,1\}^{d-1})\,.
\end{equation}
From the discussion above, we know that every $2$-level polytope $P$ that has a facet isomorphic to $P_0$ satisfies $P = \conv (\vertex(P_0) \cup A)$ for some $A \subseteq \firstX$ with $e_1 \in A$, in some $\mathcal{H}$-embedding.

\begin{ex}\label{ex:first_X}
Let $d = 3$, let $P_0$ be the $2$-simplex, and let $\Gamma_0$ be any of its simplicial cores (there is just one, up to symmetry). Using \eqref{eq:first_X}, it is easy to compute that
\begin{align*}
\firstX &= \bigcup_{c \in \{0,1\}^2} {\scriptsize\begin{pmatrix}
1 & 0 & 0 \\
c_1 & 1 & 0 \\
c_2 & 0 & 1
\end{pmatrix}^{-1}}\cdot(\{1\} \times \{0,1\}^2) \\
&=
\left\{{\scriptsize
\begin{pmatrix}1\\ -1\\ -1\end{pmatrix},
\begin{pmatrix}1\\ -1\\ 0\end{pmatrix},
\begin{pmatrix}1\\ -1\\ 1\end{pmatrix},
\begin{pmatrix}1\\ 0\\ -1\end{pmatrix},
\begin{pmatrix}1\\ 0\\ 0\end{pmatrix},
\begin{pmatrix}1\\ 0\\ 1\end{pmatrix},
\begin{pmatrix}1\\ 1\\ -1\end{pmatrix},
\begin{pmatrix}1\\ 1\\ 0\end{pmatrix},
\begin{pmatrix}1\\ 1\\ 1\end{pmatrix}
}\right\}.
\end{align*}
Notice that in this case $\firstX$ can be more compactly expressed as $\{1\} \times \{-1,0,1\}^2$. This leads to an alternative way to describe $\firstX$ in general, which is discussed in detail in Section \ref{sec:tiling}. In Figure \ref{fig:first_X}, we represent the $\mathcal H$-embedding of $P_0$ in $\{0\} \times \R^2$ with respect to $\Gamma_0$ and the set $\firstX$. 

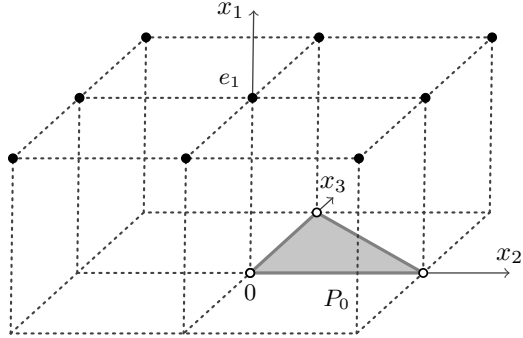
\begin{figure}[ht!]\centering
\begin{tikzpicture}[line join=round,line cap=round,rotate around y=90,rotate around z=-2,scale=2.3]
\tikzset{vertex/.style={draw,fill,circle,inner sep=1.1pt,minimum width=1pt},
frame/.style={line width=.8pt,draw=darkgray,dotted},
axes/.style={line width=.5pt,draw=darkgray,->}}
\draw[axes](0,1,0)--(0,1.5,0) node[left] at (0,1.5,0){$x_1$};
\draw[axes](1,0,0)--(1.25,0,0) node[above,yshift=-.5mm] at (1.25,0,0){$x_3$};
\draw[axes](0,0,1)--(0,0,1.5) node[above] at (0,0,1.5){$x_2$};


\draw[line width=1.2pt,gray,fill=darkgray!30](0,0,0)--(1,0,0)--(0,0,1)-- cycle;
\node [vertex,fill=white] at (0,0,0) {};
\node[label={below:$P_0$}] at (0,0,.5) {};


\draw[frame](0,0,1)--(0,1,1);
\draw[frame](0,0,0)--(0,1,0);
\draw[frame](1,0,0)--(1,1,0);
\draw[frame](1,0,1)--(1,1,1);
\draw[frame](1,0,-1)--(1,1,-1);
\draw[frame](0,0,-1)--(0,1,-1);
\draw[frame](-1,0,0)--(-1,1,0);
\draw[frame](-1,0,1)--(-1,1,1);
\draw[frame](-1,0,-1)--(-1,1,-1);

\draw[frame](0,0,-1)--(0,0,0)--(-1,0,0);

\draw[frame](0,1,-1)--(0,1,1);
\draw[frame](-1,1,0)--(1,1,0);

\draw[frame](-1,1,-1)--(1,1,-1)--(1,1,1)--(-1,1,1)--cycle;
\draw[frame](-1,0,-1)--(1,0,-1)--(1,0,1)--(-1,0,1)--cycle;

\node [vertex,fill=white] at (0,0,1) {};
\node [vertex,fill=white,label={below:$0$}] at (0,0,0) {};
\node [vertex,fill=white] at (1,0,0) {};

\node[vertex] at(1,1,1) {};
\node[vertex] at(0,1,1) {};
\node[vertex] at(-1,1,1) {};

\node[vertex] at(1,1,0) {};
\node[vertex,label={above left:$e_1$}] at(0,1,0) {};

\node[vertex] at(-1,1,0) {};
\node[vertex] at(1,1,-1) {};
\node[vertex] at(0,1,-1) {};
\node[vertex] at(-1,1,-1) {};
\end{tikzpicture}

\caption{$\mathcal H$-embedding of $P_0$ in $\{0\} \times \R^2$ with respect to its simplicial core $\Gamma_0$. The black points in $\{1\} \times \R^2$ form the corresponding set $\firstX$.}
\label{fig:first_X}
\end{figure}
\end{ex}

The sets $A \subseteq \firstX$ such that $\conv(\vertex(P_0) \cup A)$ is $2$-level satisfy stringent properties, such as those arising from Lemma~\ref{lem:embeddings}. This greatly reduces the possible choices for $A$. In other words, the enumeration algorithm does not have to consider every possible subset $A$ of $\firstX$, but can restrict to a much smaller family of subsets of $\firstX$. The next definition formalizes this.

\begin{defn}\label{def:complete}
Let $P_0$, $\Gamma_0$ and $\firstX$ be as above. A family $\mathcal{A}$ of subsets of $\firstX$ is called \emph{complete} with respect to $P_0$, $\Gamma_0$ if $e_1 \in A$ for every $A \in \mathcal A$ and every $2$-level $d$-polytope having a facet isomorphic to $P_0$ has an $\mathcal{H}$-embedding of the form $\conv(\vertex{(P_0)} \cup A)$ for some $A \in \mathcal{A}$.\end{defn}

The algorithm enumerates all \emph{candidate sets} $A \in \mathcal{A}$ for some complete family $\mathcal{A}$ and checks, for each of them, if it yields a $2$-level polytope $P \coloneqq \conv(\vertex(P_0) \cup A)$ that we did not previously find. In case the latter holds, the algorithm adds $P$ to the list $L_d$ of $d$-dimensional $2$-level polytopes.

Clearly, the strength of the algorithm relies on how accurate is our choice $\mathcal{A}$, and how efficiently we can enumerate the sets $A \in \mathcal{A}$. In Section~\ref{sec:closure}, we are going to define $\mathcal{A}$ as the collection of all closed sets for some closure operator over the ground set $\firstX$. Then, in Section~\ref{sec:reduction}, we will reduce the ground set and prove that the same closure operator applied on the smaller ground set yields a complete family.

The pseudocode of the enumeration algorithm is presented below, see Algorithm~\ref{algo:enumeration}. 
The correctness of the algorithm is a direct consequence of the discussion above, which relies on the fact that, for each $P_0 \in L_{d-1}$ with simplicial core $\Gamma_0$, $\mathcal{A} = \mathcal{A}(P_0,\Gamma_0)$ is a complete family.

\begin{algorithm}[ht!]
    \SetKwInOut{Input}{\sc Input}
    \SetKwInOut{Output}{\sc Output}
    \Input{a complete list $L_{d-1}$ of $(d-1)$-dimensional $2$-level polytopes, each
stored with a simplicial core}
    \Output{a complete list $L_d$ of $d$-dimensional
$2$-level polytopes, each stored with a simplicial core}
  \SetNlSty{textmd}{}{}
  Set $L_d \coloneqq \varnothing$\;
  \ForEach{$P_0 \in L_{d-1}$ with simplicial core $\Gamma_0 \coloneqq (F'_2,\ldots,F'_{d+1}; v_2,\ldots, v_{d+1})$}{
    Construct the $\mathcal{H}$-embedding of $P_0$ in $\{0\} \times \R^{d-1} \simeq \R^{d-1}$ 
    with respect to $\Gamma_0$\;
    Let $M_{d-1} \coloneqq M(\Gamma_0)$\;
	Construct $\firstX$ as in \eqref{eq:first_X}\;
		Let $\mathcal{A} = \mathcal{A}(P_0,\Gamma_0) \subseteq 2^{\firstX}$ be a complete family with respect to $P_0$, $\Gamma_0$ \label{step:complete}\label{step}\;
			\ForEach{\label{step2}$A \in \mathcal{A}$}{
				Let $P \coloneqq \conv (\vertex{(P_0)} \cup A)$\;\label{alg:step-ch}
				\If{$P$ is not isomorphic to any polytope in $L_d$ and is $2$-level \label{alg:step-iso}}{
				  Let $F_1 \coloneqq P_0$ and $v_1 \coloneqq e_1$\;
				  \For{$i = 2, \ldots, d+1$}{
				  	Let $F_i$ be the facet of $P$ distinct from $F_1$ s.t.\ $F_i \supseteq F'_i$\;
				  }
					Add polytope $P$ to $L_d$ with simplicial core $\Gamma \coloneqq (F_1,\ldots,F_{d+1};v_1,\ldots,v_{d+1})$\;
				}
			}

}

\caption{Enumeration algorithm} \label{algo:enumeration}
\end{algorithm}

\section{Closure operators}
\label{sec:closure}

In this section, we describe the closure operator leading to the family $\mathcal{A}$ that is used by our enumeration algorithm. First, we provide two operators, each implementing a condition that candidate sets $A \in \mathcal{A}$ have to satisfy in order to produce $2$-level polytopes. Then, the final operator is obtained by composing these two operators.

Before beginning the description of our operators, we recall the definition of closure operator. Let $\mathcal{X}$ be an arbitrary ground set. A \emph{closure operator} over $\mathcal{X}$ is a function $\cl : 2^{\mathcal{X}} \to 2^{\mathcal{X}}$ on the power set of $\mathcal{X}$ that is
\begin{enumerate}[label=(\roman*)]
\item \emph{idempotent}: $\cl(\cl(A)) = \cl(A)$ for every $A\subseteq \mathcal{X}$,
\item \emph{extensive}: $A \subseteq \cl(A)$ for every $A\subseteq \mathcal{X}$, 
\item \emph{monotone}: $A \subseteq B \Rightarrow \cl(A) \subseteq \cl(B)$ for every $A, B \subseteq \mathcal{X}$.
\end{enumerate}
A set $A \subseteq \mathcal{X}$ is said to be \emph{closed} with respect to $\cl$ provided that $\cl(A) = A$.

The enumeration of all closed sets of a given closure operator on some finite ground set is a well-studied problem arising in many areas, and in particular in formal concept analysis~\cite{GanterWille}. As part of our code, we implement Ganter's Next-Closure algorithm~\cite{Ganter87, Ganter91}, one of the best known algorithms for the enumeration of closed sets.

\subsection{Discrete convex hull}
\label{sec:dchcl}

In order to motivate our first operator, consider the sets $\mathcal{X}$ and $\mathcal{Y}$ where $\mathcal{X} \coloneqq \R^d$ and $\mathcal{Y}$ is the set of all (closed) halfspaces of $\R^d$. For a set $A \subseteq \mathcal{X}$, we can define $\mathcal{H}(A)$ as the set of all halfspaces that contain all the points of $A$. Similarly, for a set $B \subseteq \mathcal{Y}$, we can define $\mathcal{P}(B)$ as the set of all points that are contained in all the halfspaces of $B$. Now consider the operator $\cl : 2^\mathcal{X} \to 2^\mathcal{X}$ with 
\begin{equation}\label{eq:cl_intro}
\cl(A) \coloneqq \mathcal{P}(\mathcal{H}(A)). 
\end{equation}

We recall the following well known result in formal concept analysis and provide the proof for completeness.

\begin{lem}\label{lem:cl_intro}
Consider two sets $\mathcal X$ and $\mathcal Y$, and a relation $R \subseteq \mathcal X \times \mathcal Y$. Let the operator $\cl : 2^\mathcal{X} \to 2^\mathcal{X}$ be defined as 
\[
\cl(A) \coloneqq \{x \in \mathcal X \mid x R y \text{ for every } y \in A'\},
\]
where $A' \coloneqq \{y \in \mathcal Y \mid aR y \text{ for every } a \in A\}$. Then $\cl$ is a closure operator over $2^\mathcal{X}$.
\end{lem}

\begin{proof}
Take two sets $A \subseteq C \subseteq \mathcal X$.
The map $\cl$ is extensive since every $a \in A$ verifies the property that $a R b$ for every $b \in A'$, thus $A \subseteq \cl(A)$.
The monotonicity holds as, by construction, $A' \supseteq C'$, thus $\cl(A) \subseteq \cl(C)$.

Finally, $\cl$ is idempotent: notice that for $y \in \mathcal Y$, $a R y$ for every $a \in A$ iff 
$a R y$ for every $a \in \cl(A)$, therefore $(\cl(A))' = A'$. The claim follows.
\end{proof}

By Lemma \ref{lem:cl_intro}, the map in \eqref{eq:cl_intro} is a closure operator. It is obtained by composing the convex hull operator and the topological closure operator. In other words, we have $\cl(A) = \overline{\conv}(A)$ for all $A \subseteq \mathcal{X}$. 

Our first closure operator is inspired by this construction. 

Consider a $2$-level $(d-1)$-polytope $P_0$, and let $\Gamma_0$ be a simplicial core of $P_0$. As before, consider the corresponding $\mathcal{H}$-embedding of $P_0$ in $\{0\} \times \R^{d-1}$. 

In our context, the role of $\mathcal{X}$ is played by $\vertex(P_0) \cup \firstX$, where $\firstX$ is defined as before, see~\eqref{eq:first_X}. The role of $\mathcal{Y}$ is played by the collection of all \emph{slabs} $S = S(E) \coloneqq \{x \in \R^d \mid 0 \leqslant x(E) \leqslant 1\}$, where  $E \subseteq [d]$ is nonempty.
For $A \subseteq \vertex(P_0) \cup \firstX$, we define $\mathcal{E}(A)$ as the set of all hyperedges $E \subseteq [d]$ whose corresponding slab $S(E)$ contains all the points of $A$. That is, we let 
\begin{equation}\label{eq:cl_slabs}
\mathcal{E}(A) \coloneqq \{ E \subseteq [d] \mid E \neq \varnothing,\ 0 \leqslant x(E)\leqslant 1 \text{ for every }x \in A\}\,.
\end{equation}
Now, for $A \subseteq \firstX$ we let
\begin{equation}\label{eq:dchcl}
\dchcl(A) \coloneqq \{x \in \firstX \mid 0 \leqslant x(E) \leqslant 1 \text{ for every } E\in\mathcal{E}(\vertex(P_0) \cup A)\}\,.
\end{equation}
In other words, $\dchcl$ maps $A$ to the subset of $\firstX$ satisfying all pairs of hyperedge constraints that are satisfied by $\vertex(P_0) \cup A$. From the discussion above, the operator $\dchcl$ is a closure operator over the ground set $\firstX$. Notice that we always have $e_1 \in \dchcl(A)$ since it belongs to $\firstX$ and satisfies \emph{all} pairs of hyperedge constraints. 

Let $A \subseteq \firstX$. Using Lemma~\ref{lem:embeddings}, we see that $\conv(\vertex(P_0) \cup A)$ is $2$-level only if $A$ is closed for the operator $\dchcl$.

\subsection{Incompatibilities}

Here, we implement a second restriction on the choice of candidate sets $A \subseteq \firstX$, that uses further constraints coming from the $2$-levelness of $P \coloneqq \conv(\vertex(P_0) \cup A)$.

Every facet $F_0$ of $P_0$ can be uniquely extended to a facet $F$ of $P$ distinct from $P_0$. Since $P = \conv(\vertex(P_0) \cup A)$ is assumed to be $2$-level, we see that the vertices of $P$ are covered by at most two translates of $\aff(F)$, the affine hull of $F$.

In order to model this fact, we declare three points $u, v, w \in \vertex(P_0) \cup \firstX$ to be \emph{incompatible} whenever there exists a facet $F_0$ of $P_0$ such that $\aff(F_0)$ and its three translates containing $u,v$ and $w$ respectively cannot be covered by any two parallel hyperplanes other than $\{0\}\times \R^{d-1}$ and $\{1\}\times \R^{d-1}$ (see Figure~\ref{fig:inccl_first_X} for an illustration).

We use incompatibilities between triples of points in $\vertex(P_0) \cup \firstX$ to define the closure operator $\inccl$ on the power set of $\firstX$. For every $A \subseteq \firstX$, we let
\begin{equation}\label{eq:inccl}
\inccl(A) \coloneqq 
\begin{cases}
A & \text{if } \vertex(P_0) \cup A \text{ does not contain an incompatible triple,}\\
\firstX &\text{otherwise.}
\end{cases}
\end{equation}
The reader can easily check that $\inccl$ is a closure operator.

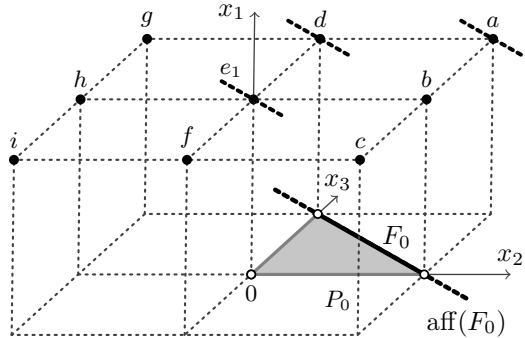
\begin{figure}[ht!]\centering
\begin{tikzpicture}[line join=round, line cap=round,rotate around y=90, rotate around z = -2,scale=2.3]
\tikzset{vertex/.style={draw,fill,circle,inner sep=1.1pt,minimum width=1pt},
frame/.style={line width=.8pt,draw=darkgray,dotted},
axes/.style={line width=.5pt,draw=darkgray,->}}
\draw[axes](0,1,0)--(0,1.5,0) node[left] at (0,1.5,0){$x_1$};
\draw[axes](1,0,0)--(1.3,0,0) node[above,yshift=-.5mm] at (1.3,0,0){$x_3$};
\draw[axes](0,0,1)--(0,0,1.5) node[above] at (0,0,1.5){$x_2$};

\draw[dotted,line width=1.7pt](1.4,0,-.4)--(-.4,0,1.4)node [below] at (-.4,0,1.4){$\mathrm{aff}(F_0)$};

\draw[line width=1.2pt,gray,fill=darkgray!30](0,0,0)--(1,0,0)--(0,0,1)-- cycle;
\node [vertex,fill=white] at (0,0,0) {};
\node[label={below:$P_0$}] at (0,0,.5) {};

\draw[line width=1.7pt](1,0,0)--(0,0,1) node [midway,right,yshift=1mm] {$F_0$};

\draw[frame](0,0,1)--(0,1,1);
\draw[frame](0,0,0)--(0,1,0);
\draw[frame](1,0,0)--(1,1,0);
\draw[frame](1,0,1)--(1,1,1);
\draw[frame](1,0,-1)--(1,1,-1);
\draw[frame](0,0,-1)--(0,1,-1);
\draw[frame](-1,0,0)--(-1,1,0);
\draw[frame](-1,0,1)--(-1,1,1);
\draw[frame](-1,0,-1)--(-1,1,-1);

\draw[frame](0,0,-1)--(0,0,0)--(-1,0,0);

\draw[frame](0,1,-1)--(0,1,1);
\draw[frame](-1,1,0)--(1,1,0);

\draw[frame](-1,1,-1)--(1,1,-1)--(1,1,1)--(-1,1,1)--cycle;
\draw[frame](-1,0,-1)--(1,0,-1)--(1,0,1)--(-1,0,1)--cycle;

\draw[dotted,line width=1.5pt](1.3,1,-.3)--(.7,1,.3);
\draw[dotted,line width=1.5pt](.3,1,-.3)--(.-.3,1,.3);
\draw[dotted,line width=1.5pt](1.3,1,.7)--(.7,1,1.3);

\node [vertex,fill=white] at (0,0,1) {};
\node [vertex,fill=white,label={below:$0$}] at (0,0,0) {};
\node [vertex,fill=white] at (1,0,0) {};

\node[vertex,label={above:$a$}] at(1,1,1) {};
\node[vertex,label={above:$b$}] at(0,1,1) {};
\node[vertex,label={above:$c$}] at(-1,1,1) {};

\node[vertex,label={above:$d$}] at(1,1,0) {};
\node[vertex,label={[yshift=1.5mm]above left:$e_1$}] at(0,1,0) {};

\node[vertex,label={above:$f$}] at(-1,1,0) {};
\node[vertex,label={above:$g$}] at(1,1,-1) {};
\node[vertex,label={above:$h$}] at(0,1,-1) {};
\node[vertex,label={above:$i$}] at(-1,1,-1) {};

\end{tikzpicture}

\caption{An example of incompatible triple of points. As in Example \ref{ex:first_X}, $P_0$ is the $2$-simplex. The facet $F_0$ of $P_0$ depicted in the figure certifies that the displayed triple $\{a,d,e_1\}$ is incompatible. Notice that $\{0,a,e_1\}$ is also incompatible.}\label{fig:inccl_first_X}
\end{figure}

\subsection{Final closure operator}
\label{sec:final_closure}

Consider the operator $\cl : 2^{\firstX} \to 2^{\firstX}$ defined as
\begin{equation}\label{eq:final_cl}
\cl(A) \coloneqq (\inccl \circ \dchcl)(A)\,.
\end{equation}
This is our final operator, which is key for the construction of the complete family used in the enumeration algorithm, as we explain in Section~\ref{sec:complete_family}.

\begin{lem}
The operator $\cl$ defined in \eqref{eq:final_cl} is a closure operator.
\end{lem}

\begin{proof}
The extensivity and monotonicity of $\cl$ directly follow from the analogous properties of $\inccl$ and $\dchcl$. Moreover, the operator $\cl$ is idempotent. Indeed, for every $A \subseteq \firstX$, if $\cl(A) = \firstX$ then $\cl(\cl(A))= \cl(A)$. Otherwise $\cl(A) = \dchcl(A)$, and $\cl(\cl(A)) = \inccl(\dchcl(\dchcl(A))) = \inccl(\dchcl(A)) = \cl(A)$.
\end{proof}

\section{Reductions of the ground set and complete family}
\label{sec:reduction}

As we described in Section \ref{sec:enum-algorithm}, the task of enumerating all $2$-level $d$-polytopes is subdivided in the subtasks of enumerating all $2$-level $d$-polytopes $P$ with a prescribed base $P_0$, for every $2$-level $(d-1)$-polytope $P_0$.
It turns out that prescribing $P_0$ as a facet yields more constraints on the structure of the entire polytope $P$. 

First, in Section \ref{sec:incompatibilities_reduction}, we look at facets of $P_0$, whose expression is well known (see Lemma \ref{lem:embeddings}), and we point out that their possible extensions restricts the choice of points of the ground set $\firstX$. 

Later, in Section \ref{sec:tiling}, we introduce a subdivision of $\firstX$ in tiles (see \eqref{eq:tile}). Then we prove that there exists a collection of translations in the hyperplane $\{1\} \times \R^{d-1}$ that move candidate sets across tiles (Lemma \ref{lem:tiling}) and preserve the property of being closed for $\dchcl$ (Lemma \ref{lem:dch_inv_tra}) and also for $\inccl$ (Lemma \ref{lem:inc_inv_tra}), thus for their composition $\cl$.

Both these arguments lead to  the construction of a reduced ground set $\finalX \subseteq \firstX$, defined in \eqref{eq:final_X}, which serves the purpose of the enumeration algorithm: as we prove in Section \ref{sec:complete_family}, the collection of all closed sets for the operator $\cl$ that are contained in $\mathcal X$ is a complete family of subsets with respect to $P_0$ and its embedding transformation matrix. This result crucially improves the efficiency of the enumeration algorithm, see Section \ref{sec:exp_impl}.  

\subsection{Removing points that always cause incompatibilities}
\label{sec:incompatibilities_reduction}

Consider any facet $F_0$ of the base $P_0$, assuming as before that $P_0$ is $\mathcal H$-embedded in $\{0\} \times \R^{d-1}$ with respect to its simplicial core $\Gamma_0$. By Lemma \ref{lem:embeddings}, there exists some nonempty $E \subseteq \{2,\dots,d\}$ such that $F_0$ is defined by either $x(E) \geqslant 0$ or $x(E) \leqslant 1$ (recall that $x(E) := \sum_{i \in E} x_i$ for every $x \in \R^d$ and every $E \subseteq [d]$). 

Now, if a set $A \subseteq \firstX$ with $e_1 \in A$ contains a point $u$ such that $u(E) \notin \{-1,0,1\}$ then one can always find points $v$ and $w$ in $\vertex(P_0) \cup A$ such that $\{u,v,w\}$ is incompatible. As a matter of fact, one can always take $v = e_1$ and $w$ as any vertex of $P_0$ not on $F_0$. Indeed, the four affine spaces $\{x \in \R^d \mid x(E) = 0,\ x_1 = 0\}$, $\{x \in \R^d \mid x(E) = 1,\ x_1 = 0\}$, $\{x \in \R^d \mid x(E) = 0,\ x_1 = 1\}$ and $\{x \in \R^d \mid x(E) = u(E),\ x_1 = 1\}$ cannot be covered by two parallel hyperplanes other than $\{0\} \times \R^{d-1}$ and $\{1\} \times \R^{d-1}$. As a consequence, such points $u$ can be removed from the ground set of $\inccl$ without changing the closed sets of $\inccl$ (except for the ``full'' closed set $A = \firstX$, which does not yield a $2$-level polytope).

Let $\mathcal{F} = \mathcal{F}(P_0,\Gamma_0)$ denote the collection of nonempty subsets $E \subseteq \{2,\ldots,d\}$ such that $x(E) \geqslant 0$ or $x(E) \leqslant 1$ defines a facet of $P_0$. We define
\begin{equation}\label{eq:second_X}
\secondX = \secondX(P_0,\Gamma_0) \coloneqq \left\{
u \in \firstX \mid u(E) \in \{-1,0,1\} \text{ for every } E \in \mathcal{F}
\right\}\,.
\end{equation}

\begin{ex}\label{ex:second_X}
Let $d = 3$, let $P_0$ be the $2$-simplex, see Example \ref{ex:first_X}. Using \eqref{eq:second_X}, we deduce that
\begin{align*}
\secondX =
\left\{{\scriptsize
\begin{pmatrix}1\\ -1\\ 0\end{pmatrix},
\begin{pmatrix}1\\ -1\\ 1\end{pmatrix},
\begin{pmatrix}1\\ 0\\ -1\end{pmatrix},
\begin{pmatrix}1\\ 0\\ 0\end{pmatrix},
\begin{pmatrix}1\\ 0\\ 1\end{pmatrix},
\begin{pmatrix}1\\ 1\\ -1\end{pmatrix},
\begin{pmatrix}1\\ 1\\ 0\end{pmatrix}
}\right\}.
\end{align*}
Indeed, the facet $x_2 +x_3 \leqslant 1$ of $P_0$ is such that $x_2 +x_3 = 2$ and $x_2 +x_3 = -2$ for the points ${\scriptsize
\begin{pmatrix}1\\ 1\\ 1\end{pmatrix},\scriptsize
\begin{pmatrix}1\\ -1\\ -1\end{pmatrix} } \in \firstX$ respectively, hence these points do not figure in $\secondX$. See Figure \ref{fig:second_X} for an illustration.
\begin{figure}[ht!]\centering
\begin{tikzpicture}[line join=round,line cap=round,rotate around y=90,rotate around z=-2,scale=2.3]
\tikzset{vertex/.style={draw,fill,circle,inner sep=1.1pt,minimum width=1pt},
frame/.style={line width=.8pt,draw=darkgray,dotted},
axes/.style={line width=.5pt,draw=darkgray,->}}
\draw[axes](0,1,0)--(0,1.5,0) node[left] at (0,1.5,0){$x_1$};
\draw[axes](1,0,0)--(1.25,0,0) node[above,yshift=-.5mm] at (1.25,0,0){$x_3$};
\draw[axes](0,0,1)--(0,0,1.5) node[above] at (0,0,1.5){$x_2$};


\draw[line width=1.2pt,gray,fill=darkgray!30](0,0,0)--(1,0,0)--(0,0,1)-- cycle;
\node [vertex,fill=white] at (0,0,0) {};
\node[label={below:$P_0$}] at (0,0,.5) {};


\draw[frame](0,0,1)--(0,1,1);
\draw[frame](0,0,0)--(0,1,0);
\draw[frame](1,0,0)--(1,1,0);
\draw[frame](1,0,1)--(1,1,1);
\draw[frame](1,0,-1)--(1,1,-1);
\draw[frame](0,0,-1)--(0,1,-1);
\draw[frame](-1,0,0)--(-1,1,0);
\draw[frame](-1,0,1)--(-1,1,1);
\draw[frame](-1,0,-1)--(-1,1,-1);

\draw[frame](0,0,-1)--(0,0,0)--(-1,0,0);

\draw[frame](0,1,-1)--(0,1,1);
\draw[frame](-1,1,0)--(1,1,0);

\draw[frame](-1,1,-1)--(1,1,-1)--(1,1,1)--(-1,1,1)--cycle;
\draw[frame](-1,0,-1)--(1,0,-1)--(1,0,1)--(-1,0,1)--cycle;

\node [vertex,fill=white] at (0,0,1) {};
\node [vertex,fill=white,label={below:$0$}] at (0,0,0) {};
\node [vertex,fill=white] at (1,0,0) {};

\node[vertex] at(0,1,1) {};
\node[vertex] at(-1,1,1) {};

\node[vertex] at(1,1,0) {};
\node[vertex,label={above left:$e_1$}] at(0,1,0) {};

\node[vertex] at(-1,1,0) {};
\node[vertex] at(1,1,-1) {};
\node[vertex] at(0,1,-1) {};
\end{tikzpicture}

\caption{The black points in the hyperplane $\{1 \} \times \R^2$ represent $\secondX$ when $P_0$ is the $2$-simplex.}
\label{fig:second_X}
\end{figure}
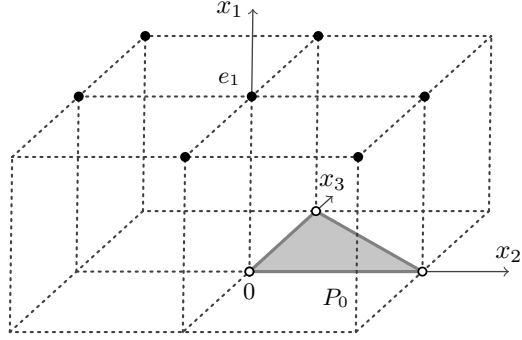
\end{ex}

With this first reduction of the ground set, we are able to simplify the description of the incompatibility closure operator. By construction, $u(E) \in \{-1,0,1\}$ for all $u \in \secondX$ and $E \in \mathcal{F}$. Let $A \subseteq \secondX$. It can be easily verified that $\vertex(P_0) \cup A$ contains no incompatible triple if and only if there is no $E \in \mathcal{F}$ and no $u, v \in A$ such that $u(E) \cdot v(E) = -1$. Thus we get for $A \subseteq \firstX$:
\begin{equation*}
\inccl(A) = 
\begin{cases}
A & \text{if } A \subseteq \secondX \text{ and there is no } E \in \mathcal{F} \text{ and no } u, v \in A \text{ such that } u(E) \cdot v(E) = -1,\\
\firstX &\text{otherwise}.
\end{cases}
\end{equation*}

\subsection{Removing points that can be avoided by translating the candidate set}
\label{sec:tiling}

We now establish useful structural properties of $\firstX$ that will lead to a different reduction of the ground set.

For each $c \in \{0,1\}^{d-1}$, we define a corresponding \emph{tile}
\begin{equation}\label{eq:tile}
T(c) \coloneqq M_d(0)^{-1} \cdot \left(\{1\} \times \prod_{i = 1}^{d-1} \{0-c_i,1-c_i\}\right)\,.
\end{equation}
Thus each tile is the vertex set of a $(d-1)$-parallelepiped in $\{1\} \times \R^{d-1}$.

\begin{lem} \label{lem:tiling}
Consider the set $\firstX$ defined as in \eqref{eq:first_X} and, for some $c \in \{0,1\}^{d-1}$ let the matrix $M_d(c)$ be defined as in  \eqref{eq:M_d} and $T(c)$ as in \eqref{eq:tile}. Then $T(c) = M_d(c)^{-1} \cdot (\{1\} \times \{0,1\}^{d-1})$. Thus $\firstX$  is the union of the $2^{d-1}$ tiles $T(c)$ for $c \in \{0,1\}^{d-1}$. Moreover, $\firstX = M_d(0)^{-1} \cdot (\{1\} \times \{-1,0,1\}^{d-1})$.
\end{lem}

\begin{proof}
Letting $c' \coloneqq - M_{d-1}^{-1}\cdot c$, we have
\begin{equation*}
M_d(c)^{-1} = 
\begin{pmatrix}
1 & \begin{matrix} 0 & \cdots & 0 \end{matrix}\\
\begin{matrix} c'_1 \\ \vdots\\ c'_{d-1} \end{matrix} & \fbox{\parbox[c][1cm][c]{1cm}{\centering$M^{-1}_{d-1}$}} \\
\end{pmatrix}\,,
\qquad
\text{and thus}
\qquad
M_d(0) \cdot M_{d}(c)^{-1} = 
\begin{pmatrix}
1 & \begin{matrix} 0 & \cdots & 0 \end{matrix}\\
\begin{matrix} -c_1 \\ \vdots\\ -c_{d-1} \end{matrix} & \fbox{\parbox[c][1cm][c]{1cm}{\centering$I$}} \\
\end{pmatrix}\,.
\end{equation*} 
Hence, we have
\begin{align*}
M_d(0) \cdot M_d(c)^{-1} \cdot (\{1\} \times \{0,1\}^{d-1})
&=\begin{pmatrix}1 \\-c\end{pmatrix} + \{0\} \times \{0,1\}^{d-1}\\
&=\{1\} \times \prod_{i = 1}^{d-1} \{0-c_i,1-c_i\}\,.
\end{align*}
By applying $M_d(0)^{-1}$, this implies immediately that $T(c) = M_d^{-1}(c) \cdot (\{1\} \times \{0,1\}^{d-1})$. Moreover, taking the union over all $c \in \{0,1\}^{d-1}$, we get $\firstX = M_d(0)^{-1} \cdot (\{1\} \times \{-1,0,1\}^{d-1})$.
\end{proof} 

As established in the next lemma, after applying the incompatibility closure operator to any subset of $\firstX$, we get a set that is either included in a tile or equal to the full ground set $\firstX$.

\begin{lem} \label{lem:inccl_tile}
Let $A \subsetneq \firstX$ be a closed set of $\inccl$. Then $A$ is contained in some tile.
\end{lem}
\begin{proof}
By contradiction, assume that $A$ is not contained in a tile. Then $A$ contains two points $u$ and $v$ such that $(M_d(0) \cdot u)_i \cdot (M_d(0) \cdot v)_i = -1$ for some index $i > 1$. But then the set $E \in \mathcal{F}$ corresponding to the $(i-1)$-th facet of the simplicial core $\Gamma_0$ has $u(E) \cdot v(E) = -1$. This implies that $\inccl(A) = \firstX$, a contradiction. 
\end{proof}

In the following lemma, we introduce certain translations that when applied to a tile produce another tile.

\begin{lem} \label{lem:tile_tra}
Consider the tile $T(c)$ for some $c \in \{0,1\}^{d-1}$. Then for every $a \in T(c)$, there exists $c' \in \{0,1\}^{d-1}$ such that $T(c)+e_1-a = T(c')$.
\end{lem}

\begin{proof}
Let 
$T=T(c)$ for some $c \in \{0,1\}^{d-1}$ and fix $a \in T$. We want to prove that there exists $c'$ such that, for each $x \in T$, $x+e_1 - a \in T(c')$. The statement then follows from the fact that each tile has the same number of points and the translation is an invertible map. Fix $x \in T$ and
let $b \coloneqq M_d(0) \cdot a$ and $y \coloneqq M_d(0) \cdot x$. We can write $b_{i+1} = \overline{b}_i - c_{i}$ and $y_{i+1} = \overline{y}_i - c_{i}$ for $i \in [d-1]$, where $\overline{b}$ and $\overline{y}$ are some vectors in $\{0,1\}^{d-1}$. Then we have 
\[
y_{i+1} - b_{i+1} = (\overline{y}_i - c_{i}) - (\overline{b}_i - c_{i}) = \overline{y}_i - \overline{b}_i
\]
for $i \in [d-1]$. It follows that $y + e_1 - b \in \{1\} \times \prod_{i=1}^{d-1} \{0-\overline{b}_i,1-\overline{b}_i\}$, and thus $x + e_1 - a \in T(\overline{b})$. 
Since $\overline{b}$ only depends on $a$, the thesis follows by setting $c'=\overline{b}$.
\end{proof}

Let $\preccurlyeq$ denote the usual lexicographic order on $\R^d$: $a \preccurlyeq b$ whenever $a = b$ or there is an index $j \in [d]$ with $a_j < b_j$ and $a_i = b_i$ for all $i < j$. Below, we will use the linear ordering $\altlexeq$ on $\R^d$ defined through the linear isomorphism $x \mapsto M_d(0) \cdot x$ by
$$
a \altlexeq b \iff M_d(0) \cdot a \preccurlyeq M_d(0) \cdot b\,.
$$

Consider a subset $A$ of some tile $T$. By Lemma \ref{lem:tile_tra}, for every $a \in A$, the translate $A':=A+e_1-a$ is contained in some tile $T'$, and thus in particular contained in $\firstX$. Moreover, the $d$-polytopes $\conv(\vertex{(P_0)} \cup A)$ and $\conv(\vertex{(P_0)} \cup A')$ are affinely isomorphic. In order to eliminate some redundancies, we wish to choose $a$ such that $A'$ is contained in a smaller portion of the ground set $\firstX$. Lemma~\ref{lem:final_X} proves that this can be achieved. 

\begin{lem}\label{lem:final_X}
Let $A$ be a subset of $\firstX$ that is contained in some tile. Then there exists $a^* \in A$ such that $A + e_1 - a^*$ is contained in $\{x \in \firstX \mid e_1 \altlexeq x\}$.
\end{lem}

\begin{proof}
Let $a^*$ be the minimum of $A$ for $\altlexeq$. By Lemma \ref{lem:tile_tra}, $A + e_1 - a^*$ is contained in $\firstX$. By contradiction, assume that there exists $a \in A$ such that $a + e_1 - a^* \altlex e_1$. Thus $M_d(0) \cdot a + e_1 - M_d(0) \cdot a^* = M_d(0) \cdot (a + e_1 - a^*) \prec M_d(0) \cdot e_1 = e_1$. This implies that there exists an index $j \in [d]$ with $j > 1$ such that $(M_d(0) \cdot a)_j < (M_d(0) \cdot a^*)_j$ and $(M_d(0) \cdot a)_i = (M_d(0) \cdot a^*)_i$ for $i < j$. Hence, $a \altlex a^*$, a contradiction. The lemma follows.
\end{proof}

Lemma~\ref{lem:final_X} motivates the following definition of our final ground set, the one used in Algorithm \ref{algo:enumeration}
 (see Figure~\ref{fig:reduced_X} for an illustration):
\begin{equation}\label{eq:final_X}	
\finalX = \finalX(P_0,\Gamma_0) :=\{x \in \secondX \mid e_1 \altlexeq x\}.
\end{equation}

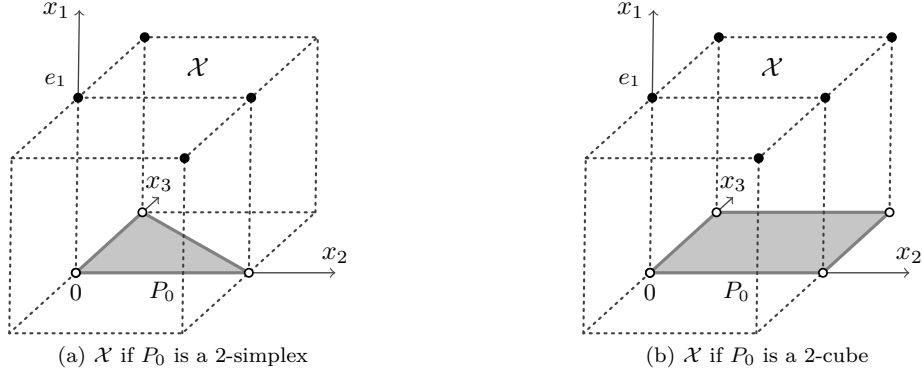
\begin{figure}[ht!]\centering
\subfloat[$\finalX$ if $P_0$ is a $2$-simplex]{\label{fig:reduced_X_a}\parbox[c][4.5cm][c]{.45\textwidth}{\centering
\begin{tikzpicture}[line join=round,line cap=round,rotate around y=90,rotate around z=-2,scale=2.3]
\tikzset{vertex/.style={draw,fill,circle,inner sep=1.1pt,minimum width=1pt},
frame/.style={line width=.8pt,draw=darkgray,dotted},
axes/.style={line width=.5pt,draw=darkgray,->}}
\draw[axes](0,1,0)--(0,1.5,0) node[left] at (0,1.5,0){$x_1$};
\draw[axes](1,0,0)--(1.25,0,0) node[above,yshift=-.5mm] at (1.25,0,0){$x_3$};
\draw[axes](0,0,1)--(0,0,1.5) node[above] at (0,0,1.5){$x_2$};


\draw[line width=1.2pt,gray,fill=darkgray!30](0,0,0)--(1,0,0)--(0,0,1)-- cycle;
\node [vertex,fill=white] at (0,0,0) {};
\node[label={[yshift=1mm]below:$P_0$}] at (0,0,.5) {};


\draw[frame](0,0,1)--(0,1,1);
\draw[frame](0,0,0)--(0,1,0);
\draw[frame](1,0,0)--(1,1,0);
\draw[frame](1,0,1)--(1,1,1);

\draw[frame](-1,0,0)--(-1,1,0);
\draw[frame](-1,0,1)--(-1,1,1);

\draw[frame](0,1,0)--(0,1,1);

\draw[frame](-1,1,0)--(1,1,0)--(1,1,1)--(-1,1,1)--cycle;
\draw[frame](0,0,0)--(-1,0,0)--(-1,0,1)--(1,0,1)--(1,0,0);

\node [vertex,fill=white] at (0,0,1) {};
\node [vertex,fill=white,label={below:$0$}] at (0,0,0) {};
\node [vertex,fill=white] at (1,0,0) {};

\node[vertex] at(0,1,1) {};
\node[vertex] at(-1,1,1) {};

\node[vertex] at(1,1,0) {};
\node[vertex,label={above left:$e_1$}] at(0,1,0) {};

\node at(.5,1,.5) {$\finalX$};

\end{tikzpicture}
}}\hspace{5mm}
\subfloat[$\finalX$ if $P_0$ is a $2$-cube]{\label{fig:reduced_X_b}\parbox[c][4.5cm][c]{.45\textwidth}{\centering
\begin{tikzpicture}[line join=round,line cap=round,rotate around y=90,rotate around z=-2,scale=2.3]
\tikzset{vertex/.style={draw,fill,circle,inner sep=1.1pt,minimum width=1pt},
frame/.style={line width=.8pt,draw=darkgray,dotted},
axes/.style={line width=.5pt,draw=darkgray,->}}
\draw[axes](0,1,0)--(0,1.5,0) node[left] at (0,1.5,0){$x_1$};
\draw[axes](1,0,0)--(1.25,0,0) node[above,yshift=-.5mm] at (1.25,0,0){$x_3$};
\draw[axes](0,0,1)--(0,0,1.5) node[above] at (0,0,1.5){$x_2$};


\draw[line width=1.2pt,gray,fill=darkgray!30](0,0,0)--(1,0,0)--(1,0,1)--(0,0,1)-- cycle;
\node [vertex,fill=white] at (0,0,0) {};
\node[label={[yshift=1mm]below:$P_0$}] at (0,0,.5) {};


\draw[frame](0,0,1)--(0,1,1);
\draw[frame](0,0,0)--(0,1,0);
\draw[frame](1,0,0)--(1,1,0);
\draw[frame](1,0,1)--(1,1,1);

\draw[frame](-1,0,0)--(-1,1,0);
\draw[frame](-1,0,1)--(-1,1,1);

\draw[frame](0,1,0)--(0,1,1);

\draw[frame](-1,1,0)--(1,1,0)--(1,1,1)--(-1,1,1)--cycle;
\draw[frame](0,0,0)--(-1,0,0)--(-1,0,1)--(0,0,1);

\node [vertex,fill=white] at (0,0,1) {};
\node [vertex,fill=white] at (1,0,1) {};
\node [vertex,fill=white,label={below:$0$}] at (0,0,0) {};
\node [vertex,fill=white] at (1,0,0) {};

\node[vertex] at(1,1,1) {};
\node[vertex] at(0,1,1) {};
\node[vertex] at(-1,1,1) {};

\node[vertex] at(1,1,0) {};
\node[vertex,label={above left:$e_1$}] at(0,1,0) {};

\node at(.5,1,.5) {$\finalX$};

\end{tikzpicture}
}}

\caption{$\mathcal H$-embedding of $P_0$ in $\{0\} \times \R^2$ with respect to its simplicial core $\Gamma_0$, together with the corresponding ground set $\finalX = \finalX(P_0,\Gamma_0)$ in $\{1\} \times \R^2$ (indicated by black points).}
\label{fig:reduced_X}
\end{figure}

We establish an invariance property of the closure operator $\dchcl$ under translations, that will be useful to prove that the closed sets of the restriction of the closure operator $\cl$ to $\finalX$ form a complete family.

\begin{lem} \label{lem:dch_inv_tra}
Consider a set $A \subseteq \firstX$ that is contained in some tile. If $A$ is closed for $\dchcl$, then $A+e_1-a$ is closed for $\dchcl$ for every $a \in A$.
\end{lem}

\begin{proof} 
For convenience, define $A' \coloneqq A + e_1 - a$. By Lemma~\ref{lem:tile_tra}, $A'$ is also contained in some tile. In particular, $A' \subseteq \firstX$. 

First, we establish a bijection between $\mathcal{E}(\vertex(P_0) \cup A)$ and $\mathcal{E}(\vertex(P_0) \cup A')$. For $E \in \mathcal{E}(\vertex(P_0) \cup A)$, we let $E' \coloneqq E$ if $(e_1-a)(E) = 0$ and $E' \coloneqq E \vartriangle \{1\}$ if $(e_1-a)(E) \neq 0$, where $\vartriangle$ denotes symmetric difference. Notice that $(e_1-a)(E) \in \{-1,0,1\}$. Moreover, $a(E) \in \{0,1\}$ since $E \in \mathcal E(\vertex(P_0) \cup A)$. Hence $(e_1-a)(E) = -1$ iff $1 \notin E$ and $a(E) = 1$, and $(e_1-a)(E) = 1$ iff $1 \in E$ and $a(E) = 0$. Observe that we always have $(e_1-a)(E) = (e_1-a)(E')$.

Take $E \in \mathcal{E}(\vertex(P_0) \cup A)$. For all $x \in \firstX$, we find
\begin{equation}
\label{eq:tra_main}
(x+e_1-a)(E') = x(E') + (e_1-a) (E')
= x(E') + (e_1-a)(E) = x(E) \in \{0,1\}\,.
\end{equation}
Indeed, this is obvious if $(e_1-a)(E) = 0$ since then $E' = E$. If $(e_1-a)(E) = -1$ then $E' = E \cup \{1\}$ and $x(E') = x(E)+1$. Finally, if $(e_1-a)(E)= 1$ then $E' = E \setminus \{1\}$ and $x(E') = x(E)-1$.

Moreover, for all $x \in \vertex(P_0)$, we have
\[
x(E') = x(E) \in \{0,1\}
\]
since $x \in \vertex(P_0)$ implies $x_1 = 0$. Therefore, $E' \in \mathcal{E}(\vertex(P_0) \cup A')$.

So far, we obtained a map $\psi : \mathcal{E}(\vertex(P_0) \cup A) \to \mathcal{E}(\vertex(P_0) \cup A')$ defined as $E \mapsto \psi(E) \coloneqq E'$. This map is injective since if $E, F \in \mathcal{E}(\vertex(P_0) \cup A)$ have $\psi(E) = \psi(F)$ then $E \cap \{2,\ldots,d\} = F \cap \{2,\ldots,d\}$ so that $(e_1 - a)(E) = (e_1 - a)(F)$. It follows that $E = F$. 

Since $A$ is closed for $\dchcl$, we have $e_1 \in A$ and hence $2e_1 - a \in A'$. Let $a' \coloneqq 2e_1 - a$. By applying the reasoning above to $A'$ and $A = A' + e_1 - a'$, we know that there exists an injective map from $\mathcal{E}(\vertex(P_0) \cup A')$ to $\mathcal{E}(\vertex(P_0) \cup A)$. This implies that $\psi$ is in fact a bijection. 

Now assume that $x' \in \dchcl(A')$, or in other words $x'(E') \in \{0,1\}$ for all $E' \in \mathcal{E}(\vertex(P_0) \cup A')$. Then letting $x \coloneqq x' - e_1 + a = x' + e_1 - a'$ and using \eqref{eq:tra_main}, we find that $x(E) = x'(E')$ for all $E\in \mathcal{E}(\vertex(P_0) \cup A)$, and thus $x(E) \in \{0,1\}$ for all $E\in \mathcal{E}(\vertex(P_0) \cup A)$. Hence $x = x'-e_1+a \in \dchcl(A) = A$. We deduce that $\dchcl(A')-e_1+a \subseteq A$, or equivalently, $\dchcl(A') \subseteq A+e_1-a = A'$. Using this, we conclude that $A'$ is a closed set for $\dchcl$, as we desired.
\end{proof}

A similar property is satisfied by $\inccl$. The proof directly follows from the definition of incompatible triple, since $\vertex(P_0) \cup A$ contains an incompatible triple if and only if $\vertex(P_0) \cup (A + e_1 - a)$ contains an incompatible triple.

\begin{lem} \label{lem:inc_inv_tra}
Consider a set $A \subseteq \firstX$ that is contained in some tile. If $A$ is closed for $\inccl$, then $A+e_1-a$ is closed for $\inccl$ for every $a \in A$.
\end{lem}

Lemmas \ref{lem:inccl_tile}, \ref{lem:dch_inv_tra} and \ref{lem:inc_inv_tra} imply that if a set $A \subsetneq \firstX$ is closed for the composite operator $\cl = \inccl \circ \dchcl$, then it is contained in a tile and every translate of $A$ of the form $A + e_1 - a$ for $a \in A$, is closed for $\cl$ as well. This allows us to reduce the ground set for $\cl$ to the smaller set $\mathcal X$.

\subsection{Complete family}
\label{sec:complete_family}

In this section, we finally prove that the collection of all closed sets for $\cl$ that are in $\mathcal X$ constitutes a complete family. This is the collection of subsets that the enumeration algorithm parses and tests for $2$-levelness.

\begin{lem}\label{lem:last-closure}
Let $P_0$ be a $2$-level $(d-1)$-dimensional polytope and $\Gamma_0$ a simplicial core of $P_0$. Define $\mathcal{X}$ as in \eqref{eq:final_X}. Then the collection of closed sets of $\cl = \inccl \circ \dchcl$ is a complete family of subsets of $\mathcal X$ with respect to $P_0$, $\Gamma_0$.
\end{lem}

\begin{proof}

Let $\mathcal{A}$ denote the collection of closed sets of $\cl$. It suffices to show that, for each $2$-level $d$-polytope $P$ having a facet isomorphic to $P_0$, there exists some closed set $A \subseteq \mathcal{X}$ such that $P$ is isomorphic to $\conv(\vertex(P_0) \cup A)$.

Let $P$ be some $2$-level $d$-polytope having a facet isomorphic to $P_0$. For the sake of simplicity, in order to avoid explicitly using the isomorphism between $P_0$ and the facet of $P$ isomorphic to $P_0$, we assume that $P_0$ is a facet of $P$. As was discussed earlier in Section~\ref{sec:enum-algorithm}, there exists some simplicial core $\Gamma$ of $P$ such that $\Gamma$ extends the simplicial core $\Gamma_0$ and its embedding transformation matrix $M_d= M_d(c)$ extends $M_{d-1}$ according to the identity \eqref{eq:M_d} for some $c \in \{0,1\}^{d-1}$ 

Consider the $\mathcal{H}$-embedding of $P$ defined by $\Gamma$. In order to simplify notation, we assume that $P$ coincides with this $\mathcal{H}$-embedding. Let $A$ denote the vertex set of the face of $P$ opposite to $P_0$. In other words, $A = \vertex(P) \setminus \vertex(P_0)$. By Corollary \ref{lem:possible_vertices}, we can assume that $A$ is a subset of $M_d(c)^{-1} \cdot (\{1\} \times \{0,1\}^{d-1})$. By Lemma \ref{lem:tiling}, the latter set is simply the tile $T(c)$. Thus $A$ is contained in a tile.

By Lemma \ref{lem:embeddings}, the $\mathcal{H}$-embedded $2$-level polytope $P$ is the intersection of slabs of the form $S(E)$ for some nonempty $E \subseteq [d]$. By Corollary~\ref{lem:possible_vertices}, every point in $\dchcl(A)$ is a vertex of $P$ belonging to $\{1\} \times \R^{d-1}$. This implies that $\dchcl(A) \subseteq A$ and thus $\dchcl(A) = A$. 

Since $P$ is $2$-level, $\vertex(P_0) \cup A$ cannot contain any incompatible triple. Hence, $\inccl(A) = A$ and in particular, $A \subseteq \secondX$, see the discussion in Section \ref{sec:incompatibilities_reduction}. 

Summarizing what we proved so far: $A$ is a closed set of $\cl$ that is contained in the tile $T(c)$. By Lemma~\ref{lem:final_X}, there exists $a^* \in A$ such that $A^* \coloneqq A + e_1 - a^*$ is contained in $\{x \in \firstX \mid e_1 \altlexeq x\}$. By Lemmas~\ref{lem:dch_inv_tra} and \ref{lem:inc_inv_tra}, the set $A^*$ is closed for both $\dchcl$ and $\inccl$. In particular, $A^*$ is also contained in $\secondX$. Therefore, $A^*$ is a closed set of $\cl$ that is contained in $\finalX$. To finish, observe that $P$ is isomorphic to $\conv(\vertex(P_0) \cup A^\ast)$.
\end{proof}

Notice that the family of closed sets for $\cl$ always includes $\firstX$ itself, that clearly does not correspond to a $2$-level polytope. We point out that, for reasons of efficiency, it is desirable to \emph{restrict} the operator $\cl$ to the smaller ground set $\mathcal{X}$, instead of working with $\cl$ as an operator on $\firstX$ and filter out closed sets which are not contained in $\finalX$.

\begin{ex}\label{ex:complete_families}
Figure \ref{fig:complete_families_a} represents the collection of all closed sets of the closure operator $\cl$ contained in $\mathcal{X} = \mathcal{X}(P_0,\Gamma_0)$ when $P_0$ is the $2$-simplex. The six sets in Figure~\ref{fig:complete_families_a} yield four nonisomorphic $2$-level polytopes, namely: the simplex, the square based pyramid, the triangular prism and the octahedron. 

Similarly, Figure~\ref{fig:complete_families_b} represents the collection of all closed sets of the closure operator $\cl$ contained in $\mathcal{X} = \mathcal{X}(P_0,\Gamma_0)$ when $P_0$ is the $2$-cube. The five sets depicted in Figure~\ref{fig:complete_families_b} correspond to the square based pyramid, the triangular prism, the $3$-cube, the $3$-cube minus one vertex. The latter is not a $2$-level polytope, the remaining ones are.

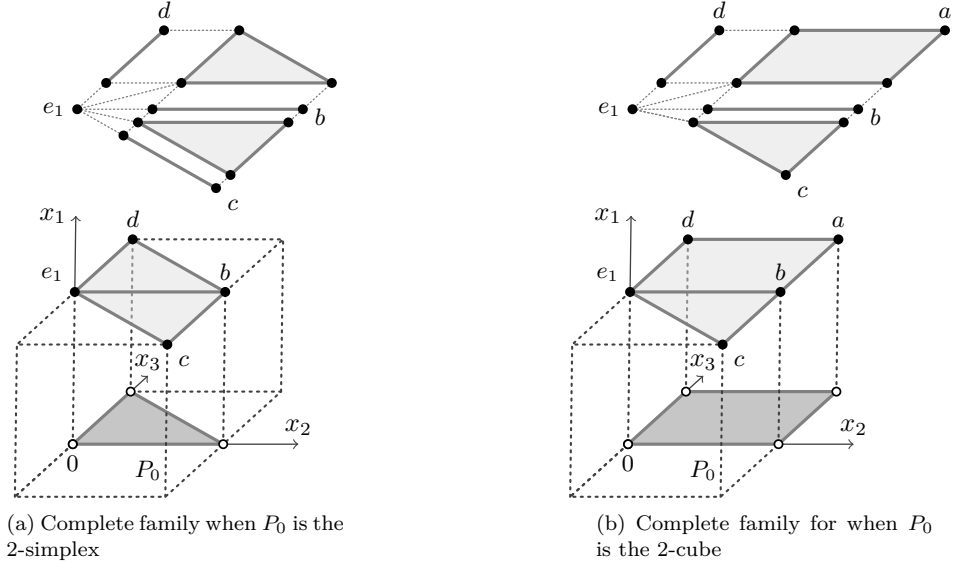
\begin{figure}[ht!]\centering
\subfloat[Complete family when $P_0$ is the $2$-simplex]{\label{fig:complete_families_a}\parbox[c][6.8cm][c]{.45\textwidth}{\centering
\begin{tikzpicture}[line join=round,line cap=round,rotate around y=90,rotate around z=-2,scale=2]
\tikzset{vertex/.style={draw,fill,circle,inner sep=1.1pt,minimum width=1pt},
frame/.style={line width=.8pt,draw=darkgray,dotted},
exploded/.style={line width=.5pt,draw=gray,densely dotted},
axes/.style={line width=.5pt,draw=darkgray,->}}

\draw[axes](1,0,0)--(1.3,0,0) node[above,yshift=-.5mm] at (1.3,0,0){$x_3$};
\draw[axes](0,0,1)--(0,0,1.5) node[above] at (0,0,1.5){$x_2$};
\draw[axes](0,1,0)--(0,1.5,0) node[left] at (0,1.5,0){$x_1$};


\draw[line width=1.2pt,gray,fill=darkgray!30](0,0,0)--(1,0,0)--(0,0,1)-- cycle;
\node [vertex,fill=white] at (0,0,0) {};
\node[label={below:$P_0$}] at (0,0,.5) {};

\draw[frame](1,0,0)--(1,1,0);

\draw[line width=1.2pt,gray,fill=lightgray!50,fill opacity=.5](1,1,0)--(0,1,1)--(-1,1,1)--(0,1,0)--cycle;

\draw[line width=1.2pt,gray](0,1,0)--(0,1,1);

\draw[frame](0,0,1)--(0,1,1);
\draw[frame](1,0,1)--(1,1,1);
\draw[frame](0,0,0)--(0,1,0);

\draw[frame](-1,0,0)--(-1,1,0);
\draw[frame](-1,0,1)--(-1,1,1);

\draw[frame](0,0,0)--(-1,0,0);


\draw[frame](0,1,1)--(1,1,1)--(1,1,0);
\draw[frame](0,1,0)--(-1,1,0)--(-1,1,1);


\draw[exploded](-.5,2.2,.5)--(.5,2.2,.5)--(.5,2.2,0);
\draw[exploded](.5,2.2,0)--(0,2.2,0)--(0,2.2,.5);
\draw[exploded](.5,2.2,.5)--(0,2.2,0)--(-.5,2.2,.5);
\draw[exploded](0,2.2,0)--(-.25,2.2,.5);
\draw[exploded](1.5,2.2,.5)--(1.5,2.2,0);
\draw[exploded](.5,2.2,1.5)--(-1.5,2.2,1.5);

\node [vertex,label={left:$e_1$}] at (0,2.2,0) {};

\draw[line width=1.2pt,gray](.5,2.2,0)--(1.5,2.2,0);
\node [vertex] at (.5,2.2,0) {};
\node [vertex,label={above:$d$}] at (1.5,2.2,0) {};

\draw[line width=1.2pt,gray,fill=lightgray!50,fill opacity=.5](.5,2.2,.5)--(1.5,2.2,.5) --(.5,2.2,1.5)-- cycle;

\node [vertex] at (.5,2.2,.5) {};
\node [vertex] at (1.5,2.2,.5) {};
\node [vertex] at (.5,2.2,1.5) {};

\draw[line width=1.2pt,gray](0,2.2,.5)--(0,2.2,1.5);
\node [vertex] at (0,2.2,.5) {};
\node [vertex,label={[yshift=-1mm]right:$b$}] at (0,2.2,1.5) {};

\draw[line width=1.2pt,gray](-.5,2.2,.5)--(-1.5,2.2,1.5);

\draw[line width=1.2pt,gray,fill=lightgray!50,fill opacity=.5](-.25,2.2,.5)--(-.25,2.2,1.5) --(-1.25,2.2,1.5)-- cycle;

\node [vertex] at (-.25,2.2,.5) {};
\node [vertex] at (-.25,2.2,1.5) {};
\node [vertex] at (-1.25,2.2,1.5) {};

\node [vertex] at (-.5,2.2,.5) {};
\node [vertex,label={below right:$c$}] at (-1.5,2.2,1.5) {};

\draw[frame](0,0,0)--(-1,0,0)--(-1,0,1)--(1,0,1)--(1,0,0);

\node [vertex,fill=white] at (0,0,1) {};
\node [vertex,fill=white,label={below:$0$}] at (0,0,0) {};
\node [vertex,fill=white] at (1,0,0) {};

\node[vertex,label={above:$b$}] at(0,1,1) {};
\node[vertex,label={below right:$c$}] at(-1,1,1) {};

\node[vertex,label={above:$d$}] at(1,1,0) {};
\node[vertex,label={above left:$e_1$}] at(0,1,0) {};

\end{tikzpicture}
}}\hspace{2em}
\subfloat[Complete family for when $P_0$ is the $2$-cube]{\label{fig:complete_families_b}\parbox[c][6.8cm][c]{.45\textwidth}{\centering
\begin{tikzpicture}[line join=round,line cap=round,rotate around y=90,rotate around z=-2,scale=2]
\tikzset{vertex/.style={draw,fill,circle,inner sep=1.1pt,minimum width=1pt},
frame/.style={line width=.8pt,draw=darkgray,dotted},
exploded/.style={line width=.5pt,draw=gray,densely dotted},
axes/.style={line width=.5pt,draw=darkgray,->}}

\draw[axes](1,0,0)--(1.3,0,0) node[above,yshift=-.5mm] at (1.3,0,0){$x_3$};
\draw[axes](0,0,1)--(0,0,1.5) node[above] at (0,0,1.5){$x_2$};
\draw[axes](0,1,0)--(0,1.5,0) node[left] at (0,1.5,0){$x_1$};


\draw[line width=1.2pt,gray,fill=darkgray!30](0,0,0)--(1,0,0)--(1,0,1)--(0,0,1)-- cycle;
\node [vertex,fill=white] at (0,0,0) {};
\node[label={below:$P_0$}] at (0,0,.5) {};

\draw[frame](1,0,0)--(1,1,0);

\draw[line width=1.2pt,gray,fill=lightgray!50,fill opacity=.5](1,1,0)--(1,1,1)--(-1,1,1)--(0,1,0)--cycle;

\draw[line width=1.2pt,gray](0,1,0)--(0,1,1);

\draw[frame](0,0,1)--(0,1,1);
\draw[frame](1,0,1)--(1,1,1);
\draw[frame](0,0,0)--(0,1,0);

\draw[frame](-1,0,0)--(-1,1,0);
\draw[frame](-1,0,1)--(-1,1,1);

\draw[frame](0,0,0)--(-1,0,0);


\draw[frame](0,1,0)--(-1,1,0)--(-1,1,1);


\draw[exploded](-.25,2.2,.5)--(.5,2.2,.5)--(.5,2.2,0);
\draw[exploded](.5,2.2,0)--(0,2.2,0)--(0,2.2,.5);
\draw[exploded](.5,2.2,.5)--(0,2.2,0)--(-.25,2.2,.5);
\draw[exploded](0,2.2,0)--(-.25,2.2,.5);
\draw[exploded](1.5,2.2,.5)--(1.5,2.2,0);
\draw[exploded](.5,2.2,1.5)--(-1.25,2.2,1.5);

\node [vertex,label={left:$e_1$}] at (0,2.2,0) {};

\draw[line width=1.2pt,gray](.5,2.2,0)--(1.5,2.2,0);
\node [vertex] at (.5,2.2,0) {};
\node [vertex,label={above:$d$}] at (1.5,2.2,0) {};

\draw[line width=1.2pt,gray,fill=lightgray!50,fill opacity=.5](.5,2.2,.5)--(1.5,2.2,.5) --(1.5,2.2,1.5)--(.5,2.2,1.5)-- cycle;

\node [vertex] at (.5,2.2,.5) {};
\node [vertex] at (1.5,2.2,.5) {};
\node [vertex] at (.5,2.2,1.5) {};
\node [vertex,label={above:$a$}] at (1.5,2.2,1.5) {};

\draw[line width=1.2pt,gray](0,2.2,.5)--(0,2.2,1.5);
\node [vertex] at (0,2.2,.5) {};
\node [vertex,label={[yshift=-1mm]right:$b$}] at (0,2.2,1.5) {};


\draw[line width=1.2pt,gray,fill=lightgray!50,fill opacity=.5](-.25,2.2,.5)--(-.25,2.2,1.5) --(-1.25,2.2,1.5)-- cycle;

\node [vertex] at (-.25,2.2,.5) {};
\node [vertex] at (-.25,2.2,1.5) {};
\node [vertex,label={below right:$c$}] at (-1.25,2.2,1.5) {};


\draw[frame](0,0,0)--(-1,0,0)--(-1,0,1)--(0,0,1);

\node [vertex,fill=white] at (0,0,1) {};
\node [vertex,fill=white] at (1,0,1) {};

\node [vertex,fill=white,label={below:$0$}] at (0,0,0) {};
\node [vertex,fill=white] at (1,0,0) {};

\node[vertex,label={above:$a$}] at(1,1,1) {};
\node[vertex,label={above:$b$}] at(0,1,1) {};
\node[vertex,label={below right:$c$}] at(-1,1,1) {};

\node[vertex,label={above:$d$}] at(1,1,0) {};
\node[vertex,label={above left:$e_1$}] at(0,1,0) {};

\end{tikzpicture}
}}

\caption{Exploded view of the collection of the closed sets of $\cl$ included in $\mathcal{X}$, with respect to two different bases $P_0$. In the upper part of the figure, all points joined by dotted lines are identified.}
\label{fig:complete_families}
\end{figure}
\end{ex}


\section{Implementation and experimental results}
\label{sec:exp_impl}

\subsection{Implementation}
\label{sec:impl}

We implemented Algorithm~\ref{algo:enumeration} in \texttt{C++}, using the Boost Dynamic Bitset library \cite{boost_dynamic_bitset} for set manipulations, and the Boost \texttt{uBLAS} library \cite{boost_ublas} for basic linear algebra computations. Besides this, our implementation heavily relies on the \texttt{C} library \texttt{nauty} \cite{McKay2014}. We use \texttt{nauty} for rejecting every $d$-polytope $P = P(A) \coloneqq \conv(\vertex(P_0) \cup A)$ that is isomorphic to some already computed $2$-level $d$-polytope $P' \in L_{d}$, and also to test whether a given $d$-polytope $P(A)$ is $2$-level. We provide more detail about the implementation below.

\paragraph{Storing and comparing $2$-level polytopes.} As mentioned before, $2$-level polytopes $P$ are stored via their 0/1 slack matrices $S(P)$. We order the rows and columns of $S(P)$ in such a way that the upper left corner of the matrix is the preferred simplicial core. Let us call two 0/1 matrices $M_1$ and $M_2$ \emph{isomorphic} if the rows and columns of $M_1$ can be permuted to give $M_2$, that is, there exist permutation matrices $L$ and $R$ such that $L M_1 R = M_2$. In order to detect isomorphism between two $2$-level polytopes $P_1$ and $P_2$, we test whether their slack matrices $S(P_1)$ and $S(P_2)$ are isomorphic. Notice that if we have \emph{any} 0/1 matrix $M$, we can check whether it is the slack matrix of a $2$-level $(d-1)$-polytope by comparing it to each $S(P_0)$, for $P_0 \in L_{d-1}$. Similarly, we can check whether $M$ is the slack matrix of an already enumerated $2$-level $d$-polytope by using $L_d$ instead. 

Isomorphism tests of two 0/1 matrices can be efficiently performed by \texttt{nauty}. We represent each 0/1 matrix $M \in \{0,1\}^{m \times n}$ by a bipartite graph $G = G(M)$ with $m + n$ vertices, together with a $2$-coloring of its vertex set, in the obvious way. For instance, we may let $V(G) \coloneqq (\{0\} \times [m]) \cup (\{1\} \times [n])$ and $E(G) \coloneqq \{ \{(0,i),(1,j)\} \mid M_{ij} = 1 \}$. The $2$-coloring is then $\phi : V(G) \to \{0,1\} : (c,k) \mapsto c$. Now, two 0/1 matrices $M_1$ and $M_2$ are isomorphic if and only if the colored graphs $G(M_1)$ and $G(M_2)$ are isomorphic, which can be tested by \texttt{nauty}.

\paragraph{Testing for $2$-levelness.} Now we describe how, for a given $A \subseteq \finalX$, we check whether $P \coloneqq P(A)$ is a $2$-level polytope or not. Intuitively, we build a 0/1 matrix $M$ which is the slack matrix of $P$, provided that $P$ is $2$-level. For each row of this matrix $M$, we extract one submatrix of $M$, which is the slack matrix of the corresponding facet of $P$, provided that $P$ is $2$-level. Then we check that each one of these submatrices is the slack matrix of a $2$-level $(d-1)$-polytope, using $L_{d-1}$. We give a formal description in the next paragraphs. 

First, we need to recall the general notion of slack matrix of a pair polytope-polyhedron, the first nested into the second (first defined in \cite{Pashkovich12,Gillis12}).

\begin{defn}[Reduced slack matrix of a pair]
Let $P$ be a polytope and $Q$ be a polyhedron with $P \subseteq Q \subseteq \R^d$. Consider an inner description $P = \conv(\{v_1, \dots, v_n\})$ and an outer description $Q = \{x \in \R^d \mid A x \leqslant b\}$, where the system $Ax \leqslant b$ consists of the $m$ inequalities $A_1 x \leqslant b_1, \dotsc, A_m x \leqslant b_m$. 
The \emph{slack matrix} of the pair $(P,Q)$ with respect to these inner and outer descriptions is the $m \times n$ matrix $S = S(P,Q)$ with $S_{ij} \coloneqq b_i - A_i v_j$. The matrix obtained from $S$ by removing the rows whose support contains the support of some other row is called \emph{reduced slack matrix} of the pair, and denoted by $S_{\mathrm{red}}(P,Q)$. 
\end{defn}

Given $A \subseteq \finalX$, we define $P$ as before and let $Q$ be the polyhedron defined by the inequalities $x(E) \geqslant 0$ and $x(E) \leqslant 1$ for $E \in \mathcal{E}(\vertex(P_0) \cup A)$. 
Let $M = M(A) \coloneqq S_{\mathrm{red}}(P,Q)$. Observe that $M$ is the slack matrix of a polytope if and only if $P = Q$. By construction, $M$ is a $0/1$ matrix. By Lemma~\ref{lem:embeddings}, $P$ is $2$-level 
iff $M$ is the slack matrix of a polytope. Thus we can reduce testing for $2$-levelness to testing whether a $0/1$ matrix is a slack matrix. This is in fact a particular case of the problem of recognizing under what assumptions a matrix is the slack matrix of a polytope, see \cite{Gouveia13}.

We do this using a non-recursive method inspired by the recursive facet system verification algorithm described below, see Algorithm~\ref{algo:verif}. In the algorithm, we write $i \sim j$ for distinct indices $i, j \in [q]$ if there is no index $k \in [q]$ distinct from $i$ and $j$ such that $F_i \cap F_j \subseteq F_k$.

\begin{algorithm}[ht!]
	\SetKw{Return}{\bf return}
    \SetKwInOut{Input}{\sc Input}
    \SetKwInOut{Output}{\sc Output}
    \Input{some polytope $P$ with $\dim(P) \geqslant 1$ and an antichain $\{F_1,\ldots,F_q\}$ of nonempty proper faces of $P$}
    \Output{`accept' if $\{F_1,\ldots,F_q\}$ is the collection of all facets of $P$, `reject' otherwise.}
  \SetNlSty{textmd}{}{}
  \If{$\dim(P) = 1$}{
    \If{$q = 2$}{
      \Return `accept'
    }
    \Return `reject'
  }
  \If{$q < \dim(P) + 1$}{
    \Return `reject'
  }
  \For{$i = 1, \ldots, q$}{
    \If{$\dim(F_i) < \dim(P) - 1$ or {\tt fsv}$(F_i;\{F_i \cap F_j \mid i \sim j\})$ = `reject'}{
      \Return `reject'
    }
  }
\Return `accept' 
\caption{Facet system verification algorithm \texttt{fsv}$(P;\{F_1,\ldots,F_q\})$}\label{algo:verif}
\end{algorithm}

Before describing our method, we establish the correctness of Algorithm~\ref{algo:verif}. Since we use this algorithm solely as a tool to establish the correctness of our method for testing whether a $0/1$-matrix is a slack matrix, it is not crucial to specify how the polytope $P$ and the collection of faces $\{F_1,\ldots,F_q\}$ are passed to the algorithm. However, for concreteness, we may assume that $P$ is given by the set $\vertex(P) \subseteq \R^d$ of its vertices\footnote{We remark that we could have modified Algorithm~\ref{algo:verif} so that only the number of vertices and ``target dimension'' of $P$ are passed together with some abstract antichain $\{F_1,\ldots,F_q\}$ such that $|\cup_i F_i| \leqslant n$, instead of the whole vertex set realized in some $\R^d$. However, we refrained from doing this since it would make the proof of Lemma~\ref{lem:fsv} longer.} and each face $F_i$ is passed as a subset of $\vertex(P)$.

\begin{lem} \label{lem:fsv}
Given a polytope $P$ with $\dim(P) \geqslant 1$ and
a collection $\{F_1,\ldots,F_q\}$ of nonempty proper faces
of $P$, no two comparable for inclusion, Algorithm~\ref{algo:verif} correctly detects if $\{F_1,\ldots,F_q\}$ is the collection of all the facets of $P$.
\end{lem}

\begin{proof}
First, assume that $\{F_1,\ldots,F_q\}$ is the collection of all the facets of $P$. We always have $q \geqslant \dim(P) + 1$. If $\dim(P) = 1$, then $q = 2$ and the algorithm correctly 
accepts. Assume now $\dim(P) \geqslant 2$. Then for each fixed $i \in [q]$, $\{F_i \cap F_j \mid i \sim j\}$ is the collection of all the facets of $F_i$ (this follows from the diamond property, see Ziegler~\cite{Ziegler95}). The result follows by induction on the dimension of $P$.

Second, assume that $\{F_1,\ldots,F_q\}$ is not the collection of all the facets of $P$. If $q < \dim(P) + 1$ then the algorithm correctly rejects. Hence may assume that $q \geqslant \dim(P) + 1 \geqslant 3$. We may also assume that every $F_i$ is a facet of $P$, otherwise the algorithm would detect this and reject. Notice that for each fixed $i$, $\{F_i \cap F_j \mid i \sim j\}$ is an antichain of proper nonempty faces of $F_i$. Now pick $i$ such that some facet $F$ of $P$ \emph{adjacent} to $F_i$ is missing from $\{F_1,\ldots,F_q\}$. Thus $F_i \cap F$ is a facet of $F_i$, not contained in any $F_j$ for $j \neq i$. Hence, $\{F_i \cap F_j \mid i \sim j\}$ is missing a facet of $F_i$, namely, $F_i \cap F$. Again, we can induct on dimension to conclude the proof.
\end{proof}

Our method for determining whether the 0/1 matrix $M = M(A)$ is the slack matrix of a $d$-polytope is similar to Algorithm~%
\ref{algo:verif}, the input being the collection of faces $F_i \coloneqq \conv(\{v_j \mid M_{ij} = 0\})$. Instead of performing any recursive call, the method directly checks that 
$\{F_i \cap F_j \mid i \sim j\}$ is the collection of all facets of $F_i$, for all $i \in [q]$. This is done by computing the matrix of non-incidences between the faces $F_j$ with 
$j \sim i$ and the vertices of $F_i$, and testing whether some matrix isomorphic to that matrix can be found in $L_{d-1}$. 

If $M$ is a slack matrix, then this test will accept for all choices of $i$. If $M$ is not a slack matrix, then this test cannot accept for all choices of $i$, because otherwise 
Algorithm~\ref{algo:verif} would also have accepted for all choices of $i$. This would imply that $M$ is a slack matrix, by Lemma~\ref{lem:fsv}.

We point out that this test only uses combinatorial information that can be found in the non-incidence matrix $M$.
Therefore there is no need to explicitly compute the convex hull of $P_0$ and $A$ in order to determine whether $P = P(A)$ is $2$-level. This improves the Algorithm presented in~\cite{Bohn15} and \cite{Fiorini16}. See Section~\ref{sec:exp_results} for the elapsed times of the new algorithm. 

\paragraph{Generating the sets of the complete family.} We implement Ganter's Next-Closure algorithm (see, e.g., \cite{GanterWille}), which we use to enumerate all closed sets of the restriction of $\cl$, see~\eqref{eq:final_cl}, to the ground set $\finalX \subseteq \firstX$, see~\eqref{eq:final_X}. The Next-Closure algorithm generates all the closed sets one after the other in the lexicographic order, starting with the closure of the empty set, which is $\cl(\varnothing) = \{e_1\}$ in our case, and ending with $\mathcal{X}$. To find the closed set that comes right after the current closed set $A$, the Next-Closure algorithm computes at most $|\finalX \setminus A|$ closures. Notice that each time we compute the discrete convex hull closure of a set $B \subseteq \finalX$, we may record the corresponding set $\mathcal{E}(\vertex(P_0) \cup B)$, since this is information that is useful for the $2$-levelness test.

\paragraph{Further optimizations.} We discard the candidate set $A$ if the maximum number of zeros per row of $M = M(A)$ is greater than the number of vertices of the base $P_0$. In this way we avoid adding multiple times different isomorphic copies of the same $2$-level polytope to the list $L_d$. If there exists a facet having more vertices than $P_0$ and it is also $2$-level, the polytope $P$ will be constructed when that facet will be taken as base. In particular, if the base $P_0$ is the simplex, only simplicial polytopes are tested for $2$-levelness.

\begin{ex}
In order to enumerate all the $3$-dimensional $2$-level polytopes, the enumeration algorithm considers all the polytopes constructed using the closed sets in Figure \ref{fig:complete_families_a} when $P_0$ is the $2$-simplex and Figure \ref{fig:complete_families_b} when $P_0$ is the $2$-cube. Obviously some polytopes are computed twice as the base changes, for instance the square-based pyramid and the triangular prism. With the optimization described above, we construct the square base pyramid, or the triangular prism only when we take the $2$-cube as base.
\end{ex}

A final optimization concerns the $2$-levelness test: if some index $i$ is found such that the number of indices $j$ such that $j \sim i$ is less than $d$, then we can safely reject the matrix $M = M(A)$. Indeed, if $M$ was the slack matrix of a $d$-polytope, then every facet $F_i$ would have at least $d$ adjacent facets $F_j$.

\subsection{Experimental results}
\label{sec:exp_results}

As our main experimental result, we obtain a database of all 
$2$-level polytopes of dimension $d \leqslant 7$, up to 
isomorphism\footnote{The complete list of all slack matrices of combinatorial inequivalent $2$-level polytopes up to dimension $7$ is available online at \texttt{\url{http://homepages.ulb.ac.be/~mmacchia/data.html}}}. Table~\ref{tbl:2Lnum} summarizes these results regarding the number of $2$-level polytopes and interesting subclasses. 

We give the number of combinatorial types in the class of polar $2$-level polytopes, those whose polar is $2$-level. In fact, $2$-levelness is in general not preserved under the operation of taking polars, and data show that the fraction of such polytopes rapidly decreases with the dimension. 

A subset of polar $2$-level polytopes is the class of centrally symmetric $2$-level polytopes. From the analysis of the data we noticed that, among all $2$-level polytopes, the centrally symmetric ones maximize the product of number of facets and number of vertices, see Figure \ref{fig:freq_facvert_1}.

Another well known class of $2$-level polytopes are the stable set polytopes of perfect graphs. Lemma \ref{lem:simple_vertex} provides an elementary way to recognize them: they are exactly $2$-level polytopes with a simple vertex.
Table~\ref{tbl:2Lnum} also shows the number of polytopes having a simplicial facet. This is a natural property to consider, being dual to the one of having a simple vertex.

Finally, we list the number of Birkhoff polytopes, for which we refer to~\cite{Paffenholz13}. Birkhoff polytopes are a classical family of $2$-level polytopes, also known as perfect matching polytope of the complete bipartite graph.

\begin{table}[ht!]\centering
\setlength{\tabcolsep}{.9em}
\begin{tabular}{r|r|r|rrrrr}
$d$ & 0/1 & $2$-level & polar & CS & STAB & $\Delta$-f & Birk\\
\hline
3 & 8 & 5 & 4 & 2 & 4 & 4  & 4\\
4 & 192 & 19 & 12 & 4 & 11 & 12 & 11 \\
5 & 1\,048\,576 & 106 & 40 & 13 & 33 & 41 & 33 \\
6 & - & 1\,150 & 262 & 45 & 148 & 248 & 129\\
7 & - & 27\,292 & 3368 & 238 & 906 & 2\,687 & 661\\[2mm]
\end{tabular}
\caption{Comparison of numbers for combinatorially inequivalent $0/1$ polytopes~\cite{Aichholzer00}, $2$-level polytopes and their sub-classes.
polar: $2$-level polytopes whose polar is $2$-level,
CS: centrally symmetric $2$-level polytopes,
$\Delta$-f: $2$-level polytopes with one simplicial facet,
STAB: stable set polytopes of perfect graphs, 
Birk: Birkhoff polytope faces from~\cite{Paffenholz13}.
}\label{tbl:2Lnum}
\end{table}

With our latest implementation, the databases for $d \leqslant 6$
were computed in a total time of about $3$ minutes on a computer cluster\footnote{Hydra balanced cluster: \texttt{https://cc.ulb.ac.be/hpc/hydra.php} with \texttt{AMD Opteron(TM) 6134 2.3 GHz} processors.}, which improves the computational times of our previous implementations~%
\cite{Bohn15,Fiorini16}. 
However, we remark that a direct comparison of the running times is not possible because the code presented here is not a secondary implementation of the same one used in \cite{Bohn15,Fiorini16}. We were able to cut down the running time by rewriting it from scratch in \texttt{C++} and using a reduced ground set, a new closure operator (Section \ref{sec:closure}, Section \ref{sec:reduction}) and a new, combinatorial $2$-levelness test (Section~\ref{sec:impl}). 

The $d = 7$ is the first challenging case for our code. We noticed that the time to compute all $2$-level polytopes with a given base $P_0$ is sharply decreasing as a function of the number of vertices of $P_0$, see Figure~\ref{fig:vertices_vs_time}. When $P_0$ is the simplex, the computational time is maximum and close to $\frac{5}{6}$ of the total time for $d = 7$.

\begin{figure}[hb!]\centering
\includegraphics[width=.6\textwidth]{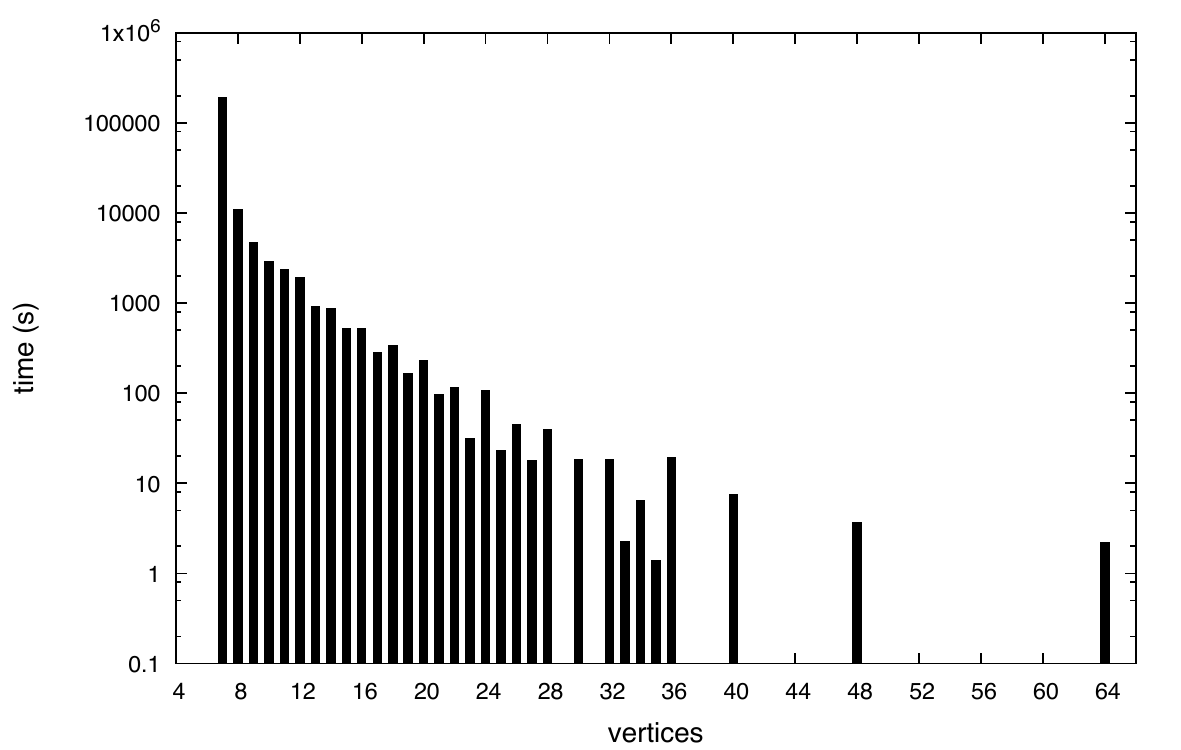}
\caption{Histogram representing the running time for the enumeration of all $2$-level $7$-polytopes with a base $P_0$ as a function of the number of vertices of $P_0$. Each bar represents the sum of the running times for bases corresponding to the same number of vertices.}\label{fig:vertices_vs_time} 
\end{figure}

Recall that our code discards candidate sets that 
give polytopes having a facet with more vertices than 
the prescribed base $P_0$. Thus the code enumerates all 
simplicial $2$-level polytopes when $P_0$ is a simplex.
In fact, it is known that the simplicial $2$-level $d$-polytopes 
are the free sums of $d/k$ simplices of dimension $k$, for $k$
a divisor of $d$~\cite{Grande15}. For instance, for $d=7$ there 
exist exactly two simplicial $2$-level $7$-polytopes: the simplex
(obtained for $k = 1$) and the cross-polytope (obtained for $k = 7$).

We split the computation into several independent jobs, each corresponding to a certain set of bases $P_0$. We created jobs testing all closed sets corresponding to only 1 base for the first 100 $2$-level $6$-dimensional bases, corresponding to 5 bases for the bases between the 101st and the 500th, corresponding to 20 bases for the bases between the 501st and the 1000th and to 50 bases for the bases between the 1001st and the 1150th (bases are ordered by increasing number of vertices). In total we submitted 208 jobs to the cluster. All jobs but the one corresponding to the $6$-simplex as base, finished in less that 3 hours. Of these jobs, all but two finished in less than 20 minutes. See Table~\ref{tab:comp_results} for more details about computational times. Notice that we could use the characterization in~\cite{Grande15} and skip the job that corresponds to taking a simplex as the base $P_0$.

The current implementation provided a list of all combinatorial types of $2$-level polytopes up to dimension $7$ in about 53 hours. There might still be ways to further improve it, for instance generalizing the closure operator and reducing the number of times isomorphic copies of the same $2$-level polytope is constructed.

\begin{table}[ht!]\centering
\setlength{\tabcolsep}{.9em}
\begin{tabular}{r|rrrr}
$d$ & $2$-level & closed sets & $2$-level tests & time (sec) \\\hline

4 & 19 & 132 & 45 & 0.034 \\
5 & 106 & 3\,828 & 456 & 1.2 \\
6 & 1\,150 & 500\,072 & 6\,875 & 205.7 \\
7 & 27\,292 & 563\,695\,419 & 159\,834 & 218\,397 \\
\end{tabular}
\caption{Computational results of enumeration algorithm (sequential time).\label{tab:comp_results}}
\end{table}

\subsection{Statistics}

Taking advantage of the data obtained, we computed a number of statistics to understand the structure and properties of $2$-level polytopes.

\begin{figure*}[ht!]\centering
\subfloat[]{\label{fig:freq_facvert_1}
\includegraphics[width=0.4\textwidth]{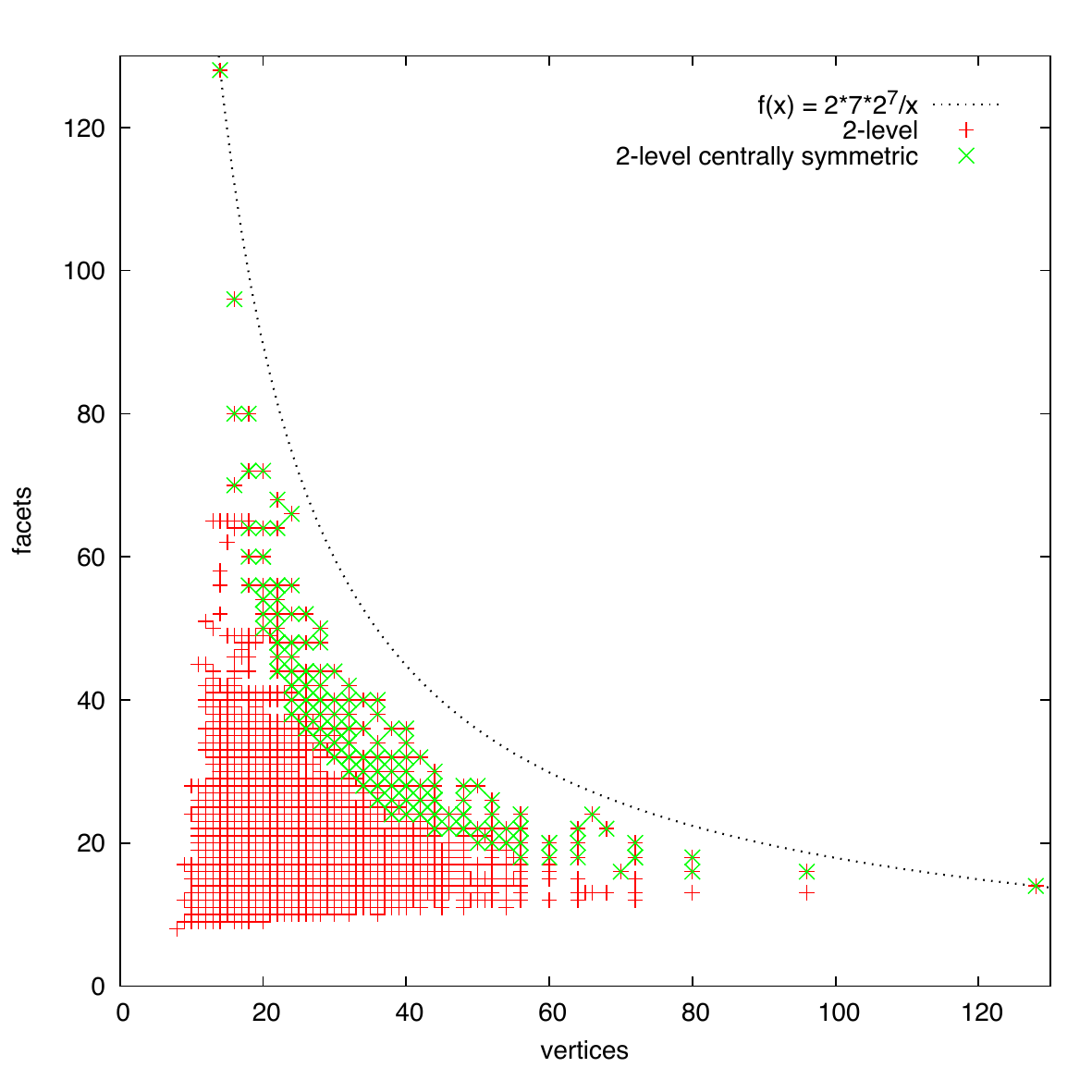}
}\hspace{2mm}%
\subfloat[]{\label{fig:freq_facvert_2}
\includegraphics[width=0.4\textwidth]{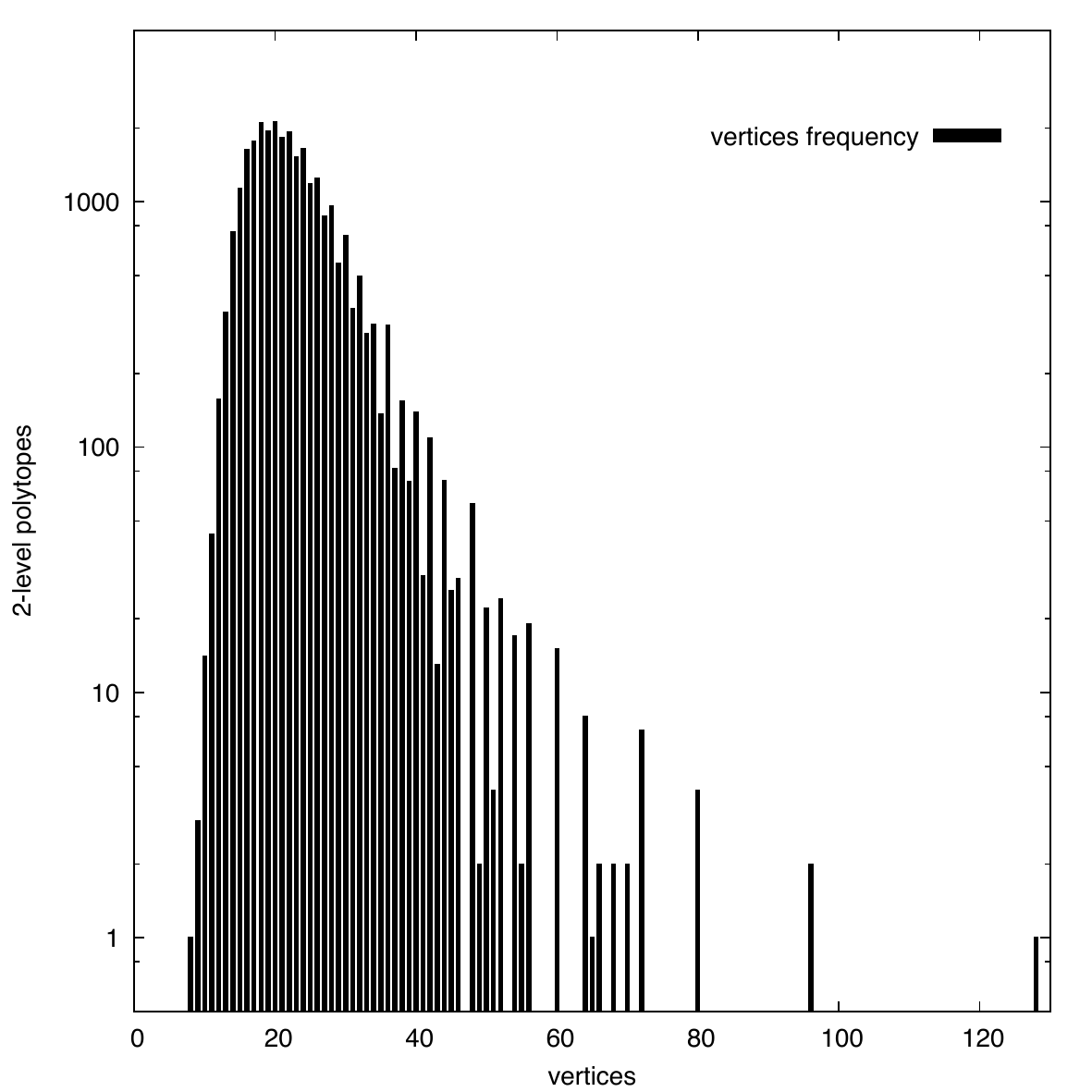}}
\caption{(a) Relation between number of facets and number of vertices for $2$-level $7$-polytopes; (b) histogram of the number of $2$-level $7$-polytopes as a function of the number of vertices.  \label{fig:freq_facvert}}
\end{figure*}

First, we considered the relation between the number of vertices and
the number of facets in $d=7$, see Figure~\ref{fig:freq_facvert}a.
The results are discussed in the next section.

Second, we inspected the number of $2$-level polytopes as a 
function of the number of vertices in dimension $7$, see 
Figure~\ref{fig:freq_facvert_2}. Interestingly, most of the
polytopes, namely 94\%, have 13 to 34 vertices.

Finally, our experiments show that all $2$-level centrally 
symmetric polytopes, up to dimension~$7$, validate Kalai's 
$3^d$ conjecture~\cite{Kalai89}. Note that for general centrally
symmetric polytopes, Kalai's conjecture is known to be true only 
up to dimension $4$~\cite{Sanyal09}. Dimension $5$ is the lowest
dimension in which we found centrally symmetric polytopes that 
are neither Hanner nor Hansen (for instance, one with $f$-vector%
\footnote{The $f$-\emph{vector} of a $d$-polytope $P$ is the $d$-dimensional vector whose $i$-th entry is the number of $(i-1)$-dimensional faces of $P$. Thus $f_0(P)$ gives the number of vertices of $P$, and $f_{d-1}(P)$ the number of facets of $P$. }
$(12, 60, 120, 90, 20)$). In dimension $6$ we found a $2$-level
centrally symmetric polytope with $f$-vector$ (20, 120, 290, 310, 
144, 24)$, for which therefore $f_0+f_4=44$. This is a stronger
counterexample to conjecture B of~\cite{Kalai89} than 
the one presented in~\cite{Sanyal09} having $f_0+f_4=48$.


\section{Discussion} \label{sec:final_remarks}

The experimental evidence we gathered leads to interesting 
research questions. As a sample, we propose three conjectures.

The first conjecture is motivated by Figure~\ref{fig:freq_facvert_1}.
\begin{conj} \label{conj:1st}
For every $2$-level $d$-polytope $P$, we have $f_0(P)f_{d-1}(P) \leqslant d2^{d+1}$
\end{conj}
Experiments show that this upper bound holds up to $d=7$. A recent work subsequent to the conference version of the current paper established that this conjecture is true for several infinite classes of $2$-level polytopes~\cite{Aprile16}.
It is known that $f_0(P) \leqslant 2^d$ with equality if and only if $P$ is a cube 
and $f_{d-1}(P) \leqslant 2^d$ with equality if and only if $P$ is a cross-polytope~%
\cite{Gouveia10}. Notice that, in both of these cases, $f_0(P)f_{d-1}(P) = d2^{d+1}$.

\begin{figure}[hb!]\centering
\subfloat[Numbers of $2$-level polytopes and $2$-level suspensions.]{\label{fig:bounds1}
\includegraphics[width=0.48\textwidth]{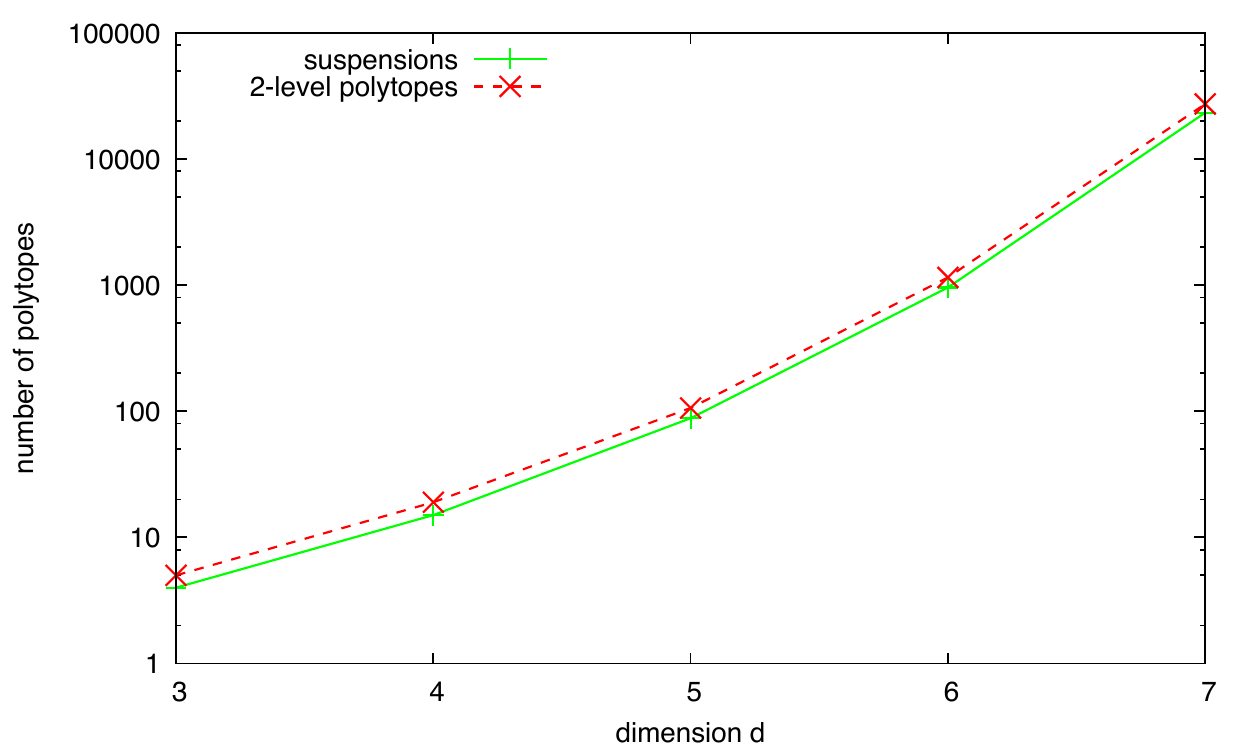}}\hspace{2mm}
\subfloat[Maximum number $f(d)$ of faces of $2$-level polytopes.]{\label{fig:bounds2}
\includegraphics[width=0.48\textwidth]{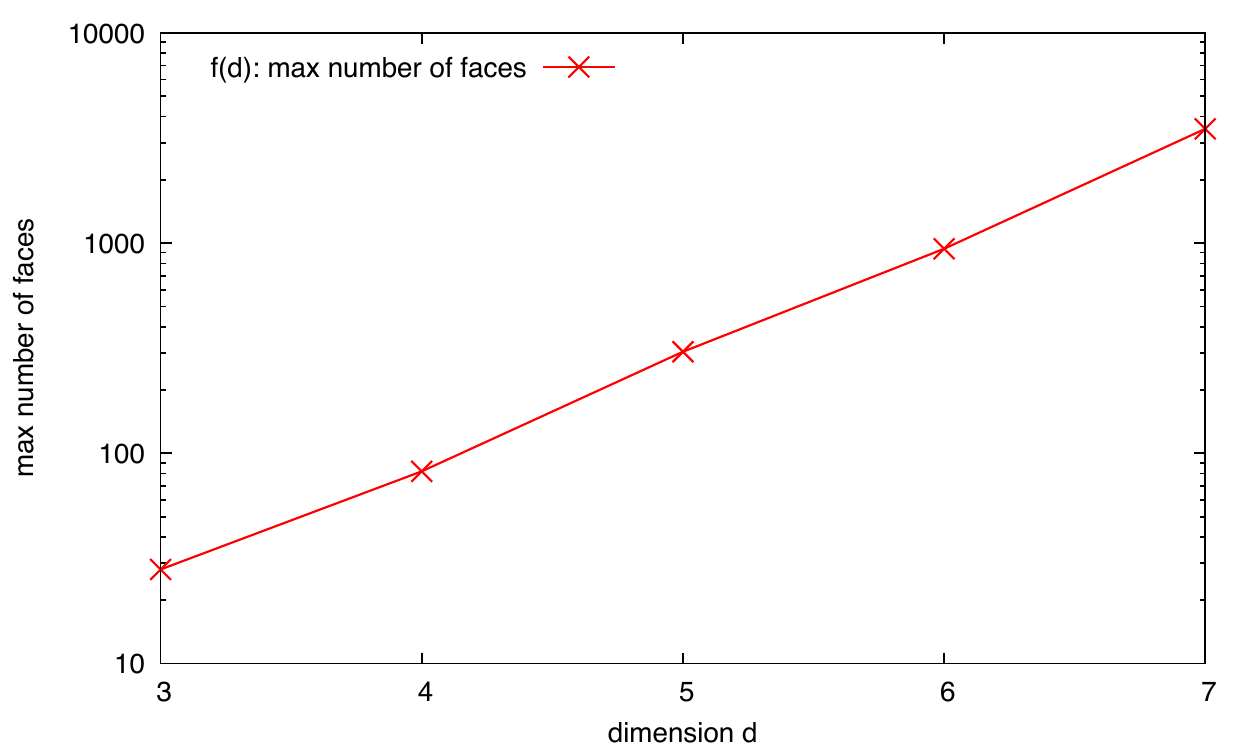}}
\caption{Bounds on number of 2-level polytopes, suspensions and faces.\label{fig:bounds}}
\end{figure}

The second conjecture concerns the asymptotic growth of the function $\ell(d)$ that 
counts the number of (combinatorially distinct) $2$-level polytopes in dimension $d$. 
All the known constructions of $2$-level polytopes are ultimately based on
graphs (sometimes directed). As a matter of fact, the best lower bound we have on the
number of $2$-level polytopes is $\ell(d) \geqslant 2^{\Omega(d^2)}$. For instance, 
stable set polytopes of bipartite graphs give $\ell(d) \geqslant 2^{d^2/4-o(1)}$.
This motivates our second conjecture.

\begin{conj} \label{conj:2nd}
The number $\ell(d)$ of combinatorially distinct $2$-level $d$-polytopes satisfies
$\ell(d) \leqslant 2^{\mathrm{poly}(d)}$.
\end{conj}

A \emph{suspension} of a polytope $P_0 \subseteq \{x \in \R^d \mid x_1 = 0\}$ is any 
polytope $P$ obtained as the convex hull of $P_0$ and $P_1$, where $P_1 \subseteq 
\{x \in \R^d \mid x_1 = 1\}$ is the translate of some non-empty face of $P_0$. For
instance, the prism and the pyramid over a polytope $P_0$ are examples of suspensions. 
Also, any stable set polytope is a suspension.

Analyzing our experimental data, we noticed that a majority of $2$-level $d$-polytopes 
for $d \leqslant 7$ are suspensions of $(d-1)$-polytopes. Let $s(d)$ denote the number 
of (combinatorially distinct) $2$-level suspensions of dimension $d$. In 
Table~\ref{tbl:2Lnums}, we give the values of the $\ell(d)$ and $s(d)$ coming 
from our experiments, for $d \leqslant 7$, see also Figure~\ref{fig:bounds1}.

\begin{table}[ht!]\centering
\setlength{\tabcolsep}{.9em}
\begin{tabular}{r|rrr}
$d$ & $\ell(d)$ & $s(d)$ & $\frac{s(d)}{\ell(d)}$ \\\hline
3 & 5 & 4 & .8 \\
4 & 19 & 15 & .789 \\
5 & 106 & 88 & .830 \\
6  & 1\,150 & 956 & .831 \\
7 & 27\,292 & 23\,279 & .854 \\
\end{tabular}
\caption{Number of 2-level suspensions $s(d)$, 2-level polytopes $\ell(d)$, ratio of number of $2$-level suspensions to $2$-level polytopes.\label{tbl:2Lnums}}
\end{table}

In view of Table~\ref{tbl:2Lnums}, it is natural to ask what is the fraction 
of $2$-level $d$-polytopes that are suspensions. Excluding dimension $3$, we 
observe that this fraction increases with the dimension. This motivates our
last (and most risky) conjecture.

\begin{conj} \label{conj:3rd}
Letting $\ell(d)$ and $s(d)$ respectively denote the number of combinatorially distinct 
$2$-level polytopes and $2$-level suspensions in dimension $d$, we have 
$\ell(d) = \Theta(s(d))$.
\end{conj}

We conclude by proving some dependence between the above conjectures.

\begin{prop}
Conjecture~\ref{conj:3rd} implies Conjecture~\ref{conj:2nd}.
\end{prop}

\begin{proof}
Let us prove by induction that Conjecture~\ref{conj:3rd} implies $\ell(d) \leqslant c^{d^3}$ 
for a sufficiently large constant $c > 1$. Let $c \geqslant 2$ be large enough so that 
$\ell(d) \leqslant c \cdot s(d)$ for all dimensions $d$. Notice that the maximum number 
$f(d)$ of faces of a $2$-level $d$-polytope satisfies $f(d) \leqslant (2^d)^d = 2^{d^2}
\leqslant c^{d^2}$ since $2$-level $d$-polytopes have at most $2^d$ vertices. Now using 
the induction hypothesis $\ell(d-1) \leqslant c^{(d-1)^3}$, we have
\begin{equation*}
\ell(d)     \leqslant c \cdot s(d)
            \leqslant c \cdot \ell(d-1) \cdot f(d-1)
            \leqslant c \cdot c^{(d-1)^3} \cdot c^{(d-1)^2}
            \leqslant c^{d^3},
\end{equation*}
which proves the claim.
\end{proof}

\section*{Acknowledgments}

We acknowledge support from the following research grants: ERC grant \emph{FOREFRONT} (grant agreement no. 615640) funded by the European Research Council under the EU's 7th Framework Programme (FP7/2007-2013), \emph{Ambizione} grant PZ00P2 154779 \emph{Tight formulations of 0-1 problems} funded by the Swiss National Science Foundation, the research grant \emph{Semidefinite extended formulations} (Semaphore 14620017) funded by F.R.S.-FNRS, and the \emph{ARC} grant AUWB-2012-12/17-ULB2 \emph{COPHYMA} funded by the French community of Belgium. 

\bibliographystyle{amsplain}
\bibliography{enumeration_arxiv}

\end{document}